\documentclass[11pt]{amsart}

\usepackage{amsfonts}
\usepackage{amsmath}
\usepackage{amssymb, esint}
\usepackage{amsthm,color}
\usepackage{enumerate}
\usepackage{mathtools}

\newcommand{\R}{{\mathbb R}}

\newcommand{\Ss}{{\mathbb S}}

\def\u{{\mathbf u}}
\def\v{{\mathbf v}}

\def\f{{\mathbf f}}
\def\vphi{{\boldsymbol\phi}}
\def\vpsi{{\boldsymbol\psi}}

\def\ee{{\mathbf e}}
\def\h{{\mathbf h}}
\def\g{{\mathbf g}}
\def\w{{\mathbf w}}

\def\a{{\mathbf a}}
\def\b{{\mathbf b}}
\def\0{{\mathbf 0}}

\newcommand{\e}{\epsilon}
\newcommand{\vp}{\varphi}

\newcommand{\ra}{\rightarrow}

\theoremstyle{plain}
\newtheorem{theorem}{Theorem}[section]

\newtheorem{lemma}[theorem]{Lemma}
\newtheorem{proposition}[theorem]{Proposition}

\theoremstyle{definition}

\theoremstyle{remark}

\newtheorem{case}{Case}

\newtheorem{remark}[theorem]{Remark}

\numberwithin{equation}{section}

\title[Isolated Singularities for Elliptic Systems]{Isolated Singularities for Semilinear Elliptic Systems \\ with Power-Law Nonlinearity}

\author{Marius Ghergu}
\address{School of Mathematics and Statistics, University College Dublin, Belfield, Dublin 4, Ireland}
\email{marius.ghergu@ucd.ie}

\author{Sunghan Kim}
\address{Department of Mathematical Sciences, Seoul National University, Seoul 08826, Korea}
\email{sunghan290@snu.ac.kr}

\author{Henrik Shahgholian}
\address{Department of Mathematics, Royal Institute of Technology, 100 44 Stockholm, Sweden}
\email{henriksh@kth.se}



\begin{document}

\maketitle

\begin{abstract}

We study the system $-\Delta \u=|\u|^{\alpha-1} \u$ with $1<\alpha\leq\frac{n+2}{n-2}$, where $\u=(u_1,\dots,u_m)$, $m\geq 1$, is a $C^2$ nonnegative function that develops an isolated singularity in a domain of $\R^n$, $n\geq 3$. Due to the multiplicity of the components of $\u$, we observe a new Pohozaev invariant other than the usual one in the scalar case, and also a new class of singular solutions provided that the new invariant is nontrivial. Aligned with the classical theory of the scalar equation, we classify the solutions on the whole space as well as the punctured space, and analyze the exact asymptotic behavior of local solutions around the isolated singularity. On the technical level, we adopt the method of the moving spheres and the balanced-energy-type monotonicity functionals. 
\end{abstract}

\tableofcontents


\section{Introduction}\label{section:intro}


\subsection{Background}

This paper concerns the analysis of singular solutions to  semilinear elliptic systems with power-law nonlinearity of type
\begin{equation}\label{eq:main}
-\Delta \u=|\u|^{\alpha-1} \u,
\end{equation}
where $1<\alpha\leq\frac{n+2}{n-2}$, and $\u = (u_1,\dots,u_m)$, $m\geq 1$, is a $C^2$ vector-valued function defined on a domain in $\R^n$, $n\geq 3$. Our primary interest is in the case when each component of $\u$ is nonnegative and the domain is of the form $B_R\setminus \{0\}$, with $B_R$ being the ball of radius $R$ centered at the origin. It is by now well known  that in cylindrical coordinates $t= -\log |x| \in \R$ and $\theta = \frac{x}{|x|}\in\Ss^{n-1}$, the transformation 
\begin{equation}\label{eq:cyl}
\u(x) = |x|^{-\frac{2}{\alpha-1}} \v\left(-\log |x|,\frac{x}{|x|}\right), 
\end{equation} 
yields the system
\begin{equation}\label{eq:main-cyl}
\partial_{tt}\v + \mu \partial_t\v + \Delta_\theta \v - \lambda \v + |\v|^{\alpha-1}\v = 0,
\end{equation}
in $(-\log R,\infty)\times\Ss^{n-1}$, and vice versa, where $\Delta_\theta$ is the Laplace-Beltrami operator on $\Ss^{n-1}$ and $\lambda$ and $\mu$ are the constants fixed throughout this paper by
\begin{equation}\label{eq:lam-mu}
\lambda = \frac{2}{\alpha-1} \left( n - 2 - \frac{2}{\alpha-1} \right),\quad \mu = \frac{4}{\alpha-1} - n + 2.
\end{equation}

The scalar case of this system was introduced by Lane \cite{Lane-1870} and later studied by Emden  \cite{Emden-1907} for describing distribution of mass densities in spherical polytropic star in hydrostatic equilibrium. Since its birth, this equation has been used in many applications such as astrophysics, kinetic theory, and quantum mechanics (see \cite{Goenner-2000}). The Lane-Emden equation has thus been  subject to intensive studies in the last few decades and nowadays  there is  a vast amount of literature treating many  aspects of the solutions to this equation and its diverse varieties.

One of the central questions\footnote{To the best of our knowledge there are three central questions in this area. The other two questions refer to the structure of singular sets (see \cite{Pacard-1993}), and non-existence theory  (see \cite{Grigoryan}, \cite{Souplet-2009}). } and a technically difficult problem for differential equations and systems  is the study of the singular solutions, that is, solutions that develop singularities. In the scalar case, the classical and subsequent works have considered the asymptotic behavior of the solutions  close to isolated singularities, with an accurate description of the asymptotic behavior of solutions around such singular points; see e.g., \cite{A2, BVV, CL2, CGS, KMPS,V} and the references therein. 

The system \eqref{eq:main} can be considered as a generalization of the Lane-Emden equation, and can also be viewed as a strongly coupled system of nonlinear Schr\"odinger equations (or more precisely the limiting system of the associated blowup solutions). In the latter point of view, there has been some development regarding classification of the global solutions, and compactness of the blowup sequence; see for instance \cite{CL, DHV} and the references therein. In the former point of view, there are many other types of generalizations, among which the Lane-Emden-Fowler systems have received considerable attention. Among possible references, we refer to \cite{BVR, BG, BVG, BM, DFF, PQS, SZ1996} for the classification of global solutions, non-existence theory of singular, positive solutions and local estimates of solutions to the Lane-Emden-Fowler systems. We refer to \cite{RZ,Z} for more general cooperative elliptic systems, and the references therein. One may also consult to the monograph \cite{DF} for a general theory regarding semilinear elliptic systems. To the best of the authors' knowledge, this is the first paper that conducts a thorough analysis on the qualitative behavior of the system \eqref{eq:main}, particularly regarding the classification of the solutions on the punctured space $\R^n\setminus\{0\}$ with respect to the balanced-energy-type functionals (subcritical case $1<\alpha<\frac{n+2}{n-2}$) and the Pohozaev identities (critical case $\alpha = \frac{n+2}{n-2}$), as well as the asymptotic behavior of local solutions around the isolated singularities. 

The key difference between the system \eqref{eq:main} and its scalar version is, of course, the multiplicity of the components. The major observation in this paper is that the system \eqref{eq:main} turns out to be very sensitive to the settimg of multiple components in the case of the upper critical exponent (that is, $\alpha = \frac{n+2}{n-2}$) and lower critical exponent (that is, $\alpha = \frac{n}{n-2}$). Specifically, in the upper critical case $\alpha = \frac{n+2}{n-2}$, we discover a new Pohozaev invariant other than the usual one. The lower critical case is rather technical and we shall present the discussion on this issue in Section \ref{subsection:asym0}. 

Let us briefly illustrate how the new Pohozaev invariant comes into play in the analysis of the system \eqref{eq:main} in the upper critical case. For the sake of clarity, let us assume that the solution $\u$ is rotationally symmetric, so that the cylindrical transformation $\v$ is a function of $t$ only. After some manipulation, one can obtain the usual Pohozaev identity, 
\begin{equation}\label{eq:brief}
\left| \frac{d\v}{dt} \right|^2 =  \frac{(n-2)^2}{4} |\v|^2 -\frac{n-2}{n} |\v|^{\frac{2n}{n-2}} + \kappa,
\end{equation}
for the system \eqref{eq:main-cyl}, with a constant $\kappa$, also known as the usual Pohozaev invariant. Due to the presence of the multiple components, we have
\begin{equation}\label{eq:2nd-poho-1}
\left| \frac{d\v}{dt} \right|^2 - \left(  \frac{d|\v|}{dt} \right)^2 = \frac{1}{|\v|^2} \sum_{1\leq i<j\leq m} \left( v_i\frac{dv_j}{dt} - v_j \frac{dv_i}{dt} \right)^2 \geq 0,
\end{equation}
and the equality on the rightmost side does not hold in general. This shows that $\kappa$ alone is insufficient to analyze the behavior of $|\v|$ completely, due to the discrepancy \eqref{eq:2nd-poho-1} between $|d\v/dt|$ and $|d|\v|/dt|$. In this paper, we find that there is another constant $\kappa_*$ such that
\begin{equation}\label{eq:|v|-poho}
\left( \frac{d |\v|}{dt} \right)^2 = \frac{(n-2)^2}{4} |\v|^2 -\frac{n-2}{n} |\v|^{\frac{2n}{n-2}} + \kappa + \frac{\kappa_*}{|\v|^2},
\end{equation}
and we shall call this constant the new Pohozaev invariant. 

Thanks to an anonymous referee, we also observe a more precise characterization of the new invariant. Multiplying by $v_i$ and $-v_j$ in the $j$-th and respectively in the $i$-th component of the system \eqref{eq:main-cyl} (with $\alpha = \frac{n+2}{n-2}$), and then adding the resulting equations together side by side, we deduce that 
\begin{equation*}
\frac{d }{dt} \left( v_i\frac{dv_j}{dt} - v_j \frac{dv_i}{dt} \right) = 0,\quad 1\leq i,j\leq m.
\end{equation*}  
Thus, there exists a constant $k_{ij}$ such that for each $1\leq i,j\leq m$ we have  
\begin{equation}\label{eq:2nd-poho-2}
 v_i\frac{dv_j}{dt} - v_j \frac{dv_i}{dt}  = k_{ij}.
\end{equation} 
Inserting \eqref{eq:2nd-poho-2} into \eqref{eq:2nd-poho-1} and comparing it with \eqref{eq:|v|-poho}, we find that 
\begin{equation}\label{eq:2nd-poho}
\kappa_* = - \sum_{1\leq i<j\leq m} k_{ij}^2.
\end{equation}
Due to such a precise characterization, we also find an explicit solution featuring $\kappa_* \neq 0$, at least in the two-particle systems (i.e., $m=2$); see Remark \ref{remark:Phi*-const3}. Without the radial symmetry, we obtain a more general formula \eqref{eq:kappa*} for the new Pohozaev invariant. 

We point out that the analysis of the behavior of solutions to system \eqref{eq:main} involve both $\kappa$ and $\kappa_*$. As surprising as it may sound, one can construct radially symmetric solutions to the two-particle system (i.e., $m=2$) having non-removable singularity, even if the associated standard Pohozaev invariant $\kappa$ is zero, see Remark \ref{remark:Phi*-const3} for full details. This is a significant difference from the case of scalar equation, where $\kappa$ fully determines the behavior of the solution around the isolated singularity, and especially $\kappa = 0$ is a sufficient and necessary condition to have removable singularity. 

Classification of solutions in higher dimensional systems (i.e., $m\geq 3$) is of independent interests and would lead to a more complete picture of this problem. Nonetheless, it is beyond the scope of this article, and we shall not push further towards this direction here. 

On the technical level, the system \eqref{eq:main} exhibits some subtleties compared to the scalar case. One of the main tools we employ in the study of \eqref{eq:main} is the method of moving spheres, which has been considered in \cite{JLX, LZ} and then continuously developed especially in the frame of the fractional Laplace operator (see, e.g., \cite{JLX2, CJSX}). The use of such a method in the case of systems requires particular attention, since the procedure can be continued in some components but should stop in others.

Another technical tool is the balanced-energy-type monotonicity functional (see e.g., \eqref{eq:Phi} below), which yields the Pohozaev identity in the upper critical case $\alpha = \frac{n+2}{n-2}$, combined with the blowup analysis. Such energy functional has been a classical tool for the study of scalar case (see, e.g., \cite{BVV, A2, KMPS} and many others). We believe that the argument presented in this paper regarding the energy functional is more effective, due to an easy observation on the scaling relation \eqref{eq:Phi-scale} that is standard in the framework of free boundary problems.


\subsection{Main Results}\label{subsection:result}

The main results are as follows. First we classify the solutions on the entire space, via the method of moving spheres.

\begin{theorem}\label{theorem:main-g} Let $\u$ be a nonnegative solution of \eqref{eq:main} in $\R^n$ with $1<\alpha\leq\frac{n+2}{n-2}$.
\begin{enumerate}[(i)]
\item If $1<\alpha<\frac{n+2}{n-2}$, $\u$ is trivial. 
\item If $\alpha = \frac{n+2}{n-2}$, then $\u$ is of the form 
\begin{equation}\label{eq:u-g}
\u (x) = (n(n-2))^{\frac{n-2}{4}} \left( \frac{r}{r^2 + |x-z|^2}\right)^{\frac{n-2}{2}} \ee,
\end{equation}
for some $z\in\R^n$, $r\geq 0$ and a unit nonnegative vector $\ee\in\R^m$. 
\end{enumerate}
\end{theorem}

\begin{remark}\label{remark:main-g}  Theorem \ref{theorem:main-g} (ii) was already proved by O. Druet, E. Hebey and J. V\'etois \cite [Proposition 1.1]{DHV} via the method of moving spheres. Here we contain the result and the proof for the reader's convenience.  
\end{remark}

Next we classify the solutions in the punctured space, through the limiting energy levels or the Pohozaev invariants of the associated energy functional and the blowup analysis, which is standard in the framework of free boundary problems. For the upper critical case $\alpha = \frac{n+2}{n-2}$, we introduce a new Pohozaev invariant, which will play the central role. 

\begin{theorem}\label{theorem:main-s} Let $\u$ be a nonnegative solution of \eqref{eq:main} in $\R^n\setminus\{0\}$ with $1< \alpha \leq\frac{n+2}{n-2}$, and let $\Phi(r,\u)$ be as in \eqref{eq:Phi} for all $r>0$.
\begin{enumerate}[(i)]
\item If $1<\alpha\leq\frac{n}{n-2}$, then $\u$ is trivial.
\item If $\frac{n}{n-2} < \alpha < \frac{n+2}{n-2}$, then $\Phi(r,\u)$ converges as $r\ra 0$ and $r\ra\infty$, and 
\begin{equation}\label{eq:class-subcrit-Phi}
\{\Phi(0+,\u), \Phi(+\infty,\u)\} \subset\left\{-\frac{\alpha-1}{\alpha+1}\lambda^{\frac{\alpha+1}{\alpha-1}},0\right\}.
\end{equation}
\begin{enumerate}[(a)]
\item $\Phi(0+,\u) = 0$, if and only if $\u$ is trivial.
\item $\Phi(+\infty,\u) = -\frac{\alpha-1}{\alpha+1}\lambda^{\frac{\alpha+1}{\alpha-1}}$, if and only if $\u$ is homogeneous of degree $-\frac{2}{\alpha-1}$, hence of the form 
\begin{equation}\label{eq:u-s}
\u(x) = \lambda^{\frac{1}{\alpha-1}} |x|^{-\frac{2}{\alpha-1}} \ee,
\end{equation}
where $\lambda$ is given by \eqref{eq:lam-mu} and $\ee\in\R^m$ is a unit nonnegative vector. 
\end{enumerate}
\item If $\alpha = \frac{n+2}{n-2}$, then $\Phi_*(r,\u)$ as in \eqref{eq:Phi*} is well-defined for all $r>0$, and there are constants $\kappa(\u)$ and $\kappa_*(\u)$ such that $\kappa(\u) = \Phi(r,\u)$ and $\kappa_*(\u) = \Phi_*(r,\u)$ for all $r>0$. Moreover, 
\begin{equation}\label{eq:class-crit-kappa}
\kappa(\u) \geq -\frac{2}{n} \left(\frac{n-2}{2}\right)^n,
\end{equation}
and 
\begin{equation}\label{eq:class-crit-kappa*}
-\left(\frac{2}{n}\left(\frac{n-2}{2}\right)^n + \kappa(\u)\right) \left(\frac{n-2}{2}\right)^{n-2}\leq \kappa_*(\u) \leq 0,
\end{equation}
where the equalities of the lower bounds of both $\kappa(\u)$ and $\kappa_*(\u)$ hold only simultaneously. 
\begin{enumerate}[(a)]
\item $\kappa(\u) = \kappa_*(\u) = 0$ if and only if $\u$ has removable singularity at the origin, hence of the form \eqref{eq:u-g}. 
\item If $\kappa(\u)^2 + \kappa_*(\u)^2 > 0$, then $\u$ has non-removable singularity at the origin, and is rotationally symmetric. Moreover, the cylindrical transformation $\v$ as in \eqref{eq:cyl} satisfies \eqref{eq:|v|-poho}. 
\item $\kappa(\u)= -\frac{2}{n} (\frac{n-2}{2})^n$ and $\kappa_*(\u) = 0$ if and only if $\u$ is homogeneous of degree $-\frac{n-2}{2}$, hence of the form 
\begin{equation}\label{eq:u-s-2}
\u(x) = \left( \frac{n-2}{2} \right)^{\frac{n-2}{2}} |x|^{-\frac{n-2}{2}} \ee,
\end{equation}
where $\ee$ is a unit nonnegative vector.
\end{enumerate}
\end{enumerate}
\end{theorem}

\begin{remark}\label{remark:Phi*-const3}
Due to the characterization \eqref{eq:2nd-poho} of the second Pohozaev invariant, one can find a rotationally symmetric solution to \eqref{eq:main-cyl} for which $\kappa_* \neq 0$, even for the case $m=2$. The following example was raised to us by an anonymous referee. Given a pair $(\kappa,\kappa_*)$ of admissible constants, satisfying the bounds \eqref{eq:class-crit-kappa} and \eqref{eq:class-crit-kappa*}, such that the equation
\begin{equation}\label{eq:Phi*-const3-1}
\frac{(n-2)^2}{4} \rho^2 - \frac{n-2}{n} \rho^{\frac{2n}{n-2}} + \kappa + \frac{\kappa_*}{\rho^2} = 0
\end{equation}
has two distinct positive roots. Following the classical work \cite{F} by R.H. Fowler, one obtains a positive, non-constant, periodic solution $\rho$ to 
\begin{equation*}
\left(\frac{d\rho}{dt}\right)^2 = \frac{(n-2)^2}{4} \rho^2 - \frac{n-2}{n} \rho^{\frac{2n}{n-2}} + \kappa + \frac{\kappa_*}{\rho^2}.
\end{equation*}
Then one can easily verify that the mapping 
\begin{equation}\label{eq:Phi*-const3-ex}
\v (t) = \rho(t) \left(  \cos \left( \sqrt{-\kappa_*} \int_{t_0}^t \frac{ds}{\rho^2(s)}\right),  \sin \left( \sqrt{-\kappa_*} \int_{t_0}^t \frac{ds}{\rho^2(s)}\right) \right)
\end{equation}
solves the system \eqref{eq:main-cyl} with $\alpha = \frac{n+2}{n-2}$, $n\geq 3$ and $m = 2$, having $\kappa$ and $\kappa_*$ as the first and respectively the second Pohozaev invariant. 

Note that $\rho$ oscillates between two positive values, which implies that the corresponding solution $\u$ under the reverse cylindrical transformation \eqref{eq:cyl} has non-removable singularity at the origin. It should be stressted out that one can also have two distinct positive roots to \eqref{eq:Phi*-const3-1} even when $\kappa=0$, provided that $\kappa_*<0$ with $|\kappa_*|$ small. This shows the existence of singular solutions having trivial, standard Pohozaev invariant. 
\end{remark}

The subsequent theorems are concerned with the local solutions in the punctured unit ball. First we deduce the asymptotic radial symmetry, by the combination of the method of moving spheres and moving plane; a similar argument appears in \cite[Theorem 1.2]{CJSX}. This result is particularly important to define the second Pohozaev invariant for local solutions. 

\begin{theorem}\label{theorem:asym-rad} Let $\u$ be a nonnegative solution of \eqref{eq:main} in $B_1\setminus\{0\}$ with $1<\alpha\leq\frac{n+2}{n-2}$. Then 
\begin{equation}\label{eq:asym-rad}
\u(x) = (1+O(|x|)) \bar\u(|x|)\quad\text{as }x\ra 0,
\end{equation}
where $\bar\u(r)$ is the average of $\u$ over $\partial B_r$. 
\end{theorem}

Utilizing the classification of solutions in the punctured space and the asymptotic radial symmetry, we obtain the exact asymptotic behavior of local solutions around the singularity. 

\begin{theorem}\label{theorem:main} Let $\u$ be a nonnegative solution of \eqref{eq:main} in $B_1\setminus\{0\}$ with $1<\alpha\leq \frac{n+2}{n-2}$. Then either $\u$ has a removable singularity at the origin, or the following alternatives hold. 
\begin{enumerate}[(i)]
\item If $\frac{n}{n-2}<\alpha <\frac{n+2}{n-2}$, then 
\begin{equation}\label{eq:asym+sub}
|\u(x)| = (1 + o(1)) \lambda^{\frac{1}{\alpha-1}} |x|^{-\frac{2}{\alpha-1}} \quad\text{as }x\ra 0,
\end{equation}
where $\lambda$ is given as in \eqref{eq:lam-mu}.
\item If $\alpha= \frac{n+2}{n-2}$, then there are $c,C>0$ such that  
\begin{equation}\label{eq:asym+crit}
c|x|^{-\frac{n-2}{2}} \leq |\u(x)| \leq C|x|^{-\frac{n-2}{2}} \quad\text{as }x\ra 0, 
\end{equation}
where $c$ depends on $\u$ while $C$ is determined by $n$ and $m$ only.
\item If $1<\alpha<\frac{n}{n-2}$, then there are $c,C>0$ such that 
\begin{equation}\label{eq:asym-}
c|x|^{2-n} \leq |\u(x)| \leq C |x|^{2-n}\quad\text{as } x\ra 0,
\end{equation}
where both $c$ and $C$ depend on $\u$. 
\item If $\alpha = \frac{n}{n-2}$, then 
\begin{equation}\label{eq:asym0}
|\u(x)|  = (1+o(1)) \left(\frac{(n-2)^2}{2|x|^2(-\log |x|)}\right)^{\frac{n-2}{2}}\quad\text{as }x\ra 0.
\end{equation}
\end{enumerate}
\end{theorem}

The paper is organized as follows. In the next section, we present the balanced-energy-type monotonicity formula and introduce the second Pohozaev invariants for the upper critical case. In Section \ref{section:global}, we classify the solutions of \eqref{eq:main} on the whole space, proving Theorem \ref{theorem:main-g}. In Section \ref{section:sing}, we investigate the properties of the solutions on the punctured space, and present the proof of Theorem \ref{theorem:main-s}. Section \ref{section:upper} is devoted to the {\it a priori} estimates for the local solutions, which will play one of the key roles in the subsequent analysis, while we prove the asymptotic radial symmetry, Theorem \ref{theorem:asym-rad}, in Section \ref{section:asym-rad}. Finally, we derive the exact asymptotic behavior of the local solutions of \eqref{eq:main} for all $1<\alpha\leq\frac{n+2}{n-2}$, in Section \ref{section:asym}. The proof of parts (i)-(iv) in Theorem \ref{theorem:main} are presented in the end of Section \ref{subsection:asym+sub}-\ref{subsection:asym0}, respectively.


\subsection{Notation and Terminology}\label{subsection:notation}

We say that $\u$ has a removable singularity at the origin, provided that $|\u|$ is bounded in any neighborhood of origin. We say that $\u$ has a non-removable singularity at the origin, if $\u$ does not have a removable singularity at the origin, that is, $|\u|$ (but not necessarily all the components of $\u$) is unbounded in any neighborhood of the origin. 

By $B_r(z)\subset\R^n$ ($n\geq 3$) we denote the ball of radius $r$ centered at $z$, and in case $z=0$, we shall simply write it by $B_r$. In addition, $\omega_n$ is the volume of the unit ball $B_1\subset\R^n$. Given an open set $\Omega\subset\R^n$, we shall denote by $\partial \Omega$ the topological boundary of $\Omega$. Moreover, when $\partial \Omega$ is $C^1$, $\nu$ denotes the unit normal on $\partial \Omega$ pointing towards the origin. $\nabla_\sigma$ will denote the tangential derivative on $\partial\Omega$. 

$\Ss^{n-1}$ is the unit sphere in $\R^n$, and is also identified with $\partial B_1$. Note that $n\omega_n$ is the area of $\Ss^{n-1}$. By $\nabla_\theta$ and $\Delta_\theta$ we shall write the derivative and, respectively, the Laplace-Beltrami operator on $\Ss^{n-1}$. 

Any vector in the target space $\R^m$ ($m\geq 1$) is written in bold. Given a vector $\a\in\R^m$, we denote by $a_i$ the $i$-th component of $\a$. By $|\a|$ we denote its $l^2$-norm, i.e., $|\a| = (\sum_{i=1}^m a_i^2)^{1/2}$. By $\a\geq 0$ (resp., $\a \leq 0$) or by saying that $\a$ is nonnegative (resp., nonpositive) we indicate that $a_i \geq 0$ (resp., $a_i \leq 0$) for each $1\leq i\leq m$. For two vectors $\a$ and $\b$, $\a\cdot\b = \sum_{i=1}^m a_ib_i$. Also given two vectorial $C^1$-functions $\f$ and $\g$, $\nabla\f: \nabla\g = \sum_{i=1}^m (\nabla f_i)\cdot (\nabla g_i)$. 

The constants $C,C_0,C_1,C_2,\cdots$ will always be positive, generic, determined by $n$, $m$ and $\alpha$ only, unless otherwise stated. We shall also call these constants universal. In addition, we shall fix $\lambda$, $\mu$ and $\bar\lambda$ throughout the paper as in \eqref{eq:lam-mu} and \begin{equation}\label{eq:lamb}
\bar\lambda = \frac{\alpha-1}{\alpha+1}\lambda^{\frac{\alpha+1}{\alpha-1}}.
\end{equation}


\section{Monotonicity Formula and Pohozaev Invariant}\label{section:energy}

We consider the balanced-energy-type functional
\begin{equation}\label{eq:Phi}
\begin{split}
\Phi(r,\u) &= \frac{r^{\mu+1}}{n\omega_n} \int_{\partial B_r}  \left( \left| \frac{\partial \u}{\partial \nu} - \frac{2}{(\alpha-1)r} \u \right|^2 - |\nabla_\sigma \u|^2 \right) d\sigma \\
&\quad + \frac{2r^{\mu+1}}{(\alpha+1)n\omega_n}\int_{\partial B_r} |\u|^{\alpha+1} \,d\sigma - \frac{\lambda r^{\mu-1}}{n\omega_n}  \int_{\partial B_r}  |\u|^2 d\sigma,
\end{split}
\end{equation}
where $\lambda$ and $\mu$ are given as in \eqref{eq:lam-mu}. Note that $\lambda \geq 0$ if and only if $\alpha \geq \frac{n}{n-2}$, and $\mu \geq 0$ if and only if $1<\alpha \leq \frac{n+2}{n-2}$. 

Let us introduce the scaling function
\begin{equation}\label{eq:ur}
\u_r(x) = r^{\frac{2}{\alpha-1}} \u(rx).
\end{equation}
Note that the problem \eqref{eq:main} is preserved under this scaling. That is, if $\u$ solves \eqref{eq:main} in $B_R\setminus\{0\}$ then $\u_r$ solves \eqref{eq:main} in $B_{R/r}\setminus\{0\}$. In terms of $\u_r$, one may easily observe that $\Phi$ satisfies the following scaling relation
\begin{equation}\label{eq:Phi-scale}
\Phi(rs,\u) = \Phi(s,\u_r),
\end{equation}
for any $r,s>0$.

Recall from \eqref{eq:cyl} the cylindrical transformation $\v$, in terms of which $\Phi$ can be represented as 
\begin{equation}\label{eq:Phi-cyl}
\Phi(r,\u) = \Psi(-\log r,\v),
\end{equation}
where $\Psi(t,\v)$ is given by 
\begin{equation}\label{eq:Psi}
\Psi(t,\v) = \frac{1}{n\omega_n}\int_{\Ss^{n-1}} \left(|\partial_t\v|^2 - |\nabla_\theta \v|^2 - \lambda |\v|^2 + \frac{2}{\alpha+1} |\v|^{\alpha+1} \right) \,d\theta.
\end{equation}

\begin{proposition}\label{proposition:Phi-monot} Let $\u$ be a nonnegative solution of \eqref{eq:main} in $B_R\setminus\{0\}$ with $1<\alpha\leq\frac{n+2}{n-2}$, and let $\Phi(r,\u)$ be as in \eqref{eq:Phi}. One has
\begin{equation}\label{eq:Phi'}
\frac{d}{dr}\Phi(r,\u) = \frac{2\mu r^\mu}{n\omega_n} \int_{\partial B_r} \left| \frac{\partial \u}{\partial \nu} - \frac{2}{(\alpha-1)r} \u \right|^2 \,d\sigma,
\end{equation}
where $\mu$ is given as in \eqref{eq:lam-mu}. In particular, the following are true. 
\begin{enumerate}[(i)]
\item If $1<\alpha<\frac{n+2}{n-2}$, then $\Phi(r,\u)$ is nondecreasing for $0<r<R$. Moreover, $\Phi(r,\u)$ is constant for $r_1< r< r_2$, if and only if $\u$ is homogeneous of degree $-\frac{2}{\alpha-1}$ in $B_{r_2}\setminus \bar{B}_{r_1}$, i.e.,
\begin{equation}\label{eq:u-hom}
\u(x) = |x|^{-\frac{2}{\alpha-1}} \u\left( \frac{x}{|x|}\right)\quad\text{in }B_{r_2}\setminus \bar{B}_{r_1}.
\end{equation}
\item If $\alpha = \frac{n+2}{n-2}$, then $\Phi(r,\u)$ is constant for $0<r<R$.
\end{enumerate}
\end{proposition}

\begin{proof} The computation is easy if one chooses the cylindrical coordinate. Since \eqref{eq:Phi-cyl} holds with $t=-\log r$, we have
\begin{equation*}
\begin{split}
r \dot\Phi(r,\u) &= - \Psi'(t,\v) \\
&= -\frac{2}{n\omega_n}\int_{\Ss^{n-1}} (( \partial_{tt}\v  - \lambda\v + |\v|^{\alpha-1}\v)\cdot \partial_t\v - \nabla _\theta \v: \nabla_\theta\partial_t\v )\,d\theta\\
&= -\frac{2}{n\omega_n}\int_{\Ss^{n-1}} (\partial_{tt}\v + \Delta_\theta\v -\lambda\v + |\v|^{\alpha-1}\v)\cdot\partial_t\v \,d\theta\\
&= \frac{2\mu}{n\omega_n} \int_{\Ss^{n-1}} |\partial_t\v|^2\,d\theta,
\end{split}
\end{equation*}
where $\dot\Phi$ and $\Psi'$ denote $d\Phi/dr$ and respectively $d\Psi/dt$, and the right side is evaluated at $t=-\log r$. In addition, when deriving the last equality we used \eqref{eq:main-cyl}. Rephrasing the rightmost side in terms of $\u$, we arrive at \eqref{eq:Phi'}. 

The assertion on the monotonicity of $\Phi$ is now clear from \eqref{eq:Phi'}. On the other hand, the assertion on the homogeneity can be shown as follows. We see that if $\alpha\neq \frac{n+2}{n-2}$, then one has $\mu\neq 0$. Hence, the assumption that $\Phi(r,\u)$ being constant for $r_1<r<r_2$ along with \eqref{eq:Phi'} yields that for any $r_1<r<r_2$, 
\begin{equation*}
\frac{\partial \u}{\partial \nu} = \frac{2}{(\alpha-1)r} \u \quad\text{on }\partial B_r,
\end{equation*}
where $\nu$ is the unit normal pointing towards the origin. Thus, $\u$ is homogeneous of degree $-\frac{2}{\alpha-1}$ in $B_{r_2}\setminus \bar{B}_{r_1}$. 
\end{proof}

\begin{remark}\label{remark:Phi-monot} As a matter of fact, \eqref{eq:Phi'} holds for $\alpha>\frac{n+2}{n-2}$, and hence $\Phi(r,\u)$ is nonincreasing in this case, since $\mu<0$ for $\alpha>\frac{n+2}{n-2}$. 
\end{remark}

\begin{remark}\label{remark:Phi-monot1}
For the case $\alpha = \frac{n+2}{n-2}$, we obtain from Proposition \ref{proposition:Phi-monot} (ii) a constant $\kappa(\u)$ such that 
\begin{equation}\label{eq:kappa}
\kappa(\u) = \Phi(r,\u),
\end{equation}
for any $0<r<R$. Since there is a one-to-one correspondence between the nonnegative solutions $\u$ of \eqref{eq:main} and $\v$ of \eqref{eq:main-cyl} via the cylindrical transform \eqref{eq:cyl}, we shall write $\kappa(\u)$ by $\kappa(\v)$ as well. In view of \eqref{eq:Phi-cyl}, it is clear that
\begin{equation}\label{eq:kappa-cyl}
\kappa(\v) = \Psi(t,\v),
\end{equation}
for any $t>-\log R$. We shall call $\kappa$ the first Pohozaev invariant.  
\end{remark}

Let us construct the second Pohozaev invariant in a general setting, that is without rotational symmetry. For $\alpha = \frac{n+2}{n-2}$, let us define, formally for the moment, the quantity  
\begin{equation}\label{eq:Phi*}
\begin{split}
\Phi_*(r,\u) &= \frac{1}{4}(r\dot{f}(r,\u))^2 - \frac{(n-2)^2}{4} f(r,\u)^2 - \kappa(\u) f(r,\u) \\
&\quad + \frac{n-2}{n}f(r,\u)^{\frac{2n-2}{n-2}} - 2\int_0^r \left(\frac{\rho}{n\omega_n}\int_{\partial B_\rho} |\nabla_\sigma \u|^2 \,d\sigma\right)\dot{f}(\rho,\u) \,d\rho \\
&\quad + \frac{2n-2}{n}\int_0^r \left(\frac{\rho}{n\omega_n}\int_{\partial B_\rho}  |\u|^{\frac{2n}{n-2}}\,d\sigma - f(\rho,\u)^{\frac{n}{n-2}}\right) \dot{f}(\rho,\u)\,d\rho,
\end{split} 
\end{equation}
where $\dot{f}$ denotes $df/dr$, and  
\begin{equation}\label{eq:f}
f(r,\u) = \frac{1}{n\omega_n r} \int_{\partial B_r} |\u|^2\,d\sigma.
\end{equation}

Notice that $\Phi_*(r,\u)$ is well-defined only if the last two double integrals on the right side are finite. Moreover, once $\Phi_*(r,\u)$ becomes well-defined, we may also deduce from
\begin{equation}\label{eq:rf'}
r \dot{f}(r,\u) = - \frac{2}{n\omega_n}\int_{\partial B_r} \u\cdot \left( \frac{\partial \u}{\partial \nu} - \frac{n-2}{2r} \u\right)\,d\sigma
\end{equation}
the following scaling relation of $\Phi_*$, 
\begin{equation}\label{eq:Phi*-scale}
\Phi_*(rs,\u) = \Phi_*(s,\u_r),
\end{equation}
which holds for any $r,s>0$. On the other hand, in terms of the cylindrical transformation $\v$, one has
\begin{equation}\label{eq:Phi*-cyl}
\Phi_*(r,\u) = \Psi_*(-\log r,\v),
\end{equation}
where $\Psi_*(t,\v)$ is given by
\begin{equation}\label{eq:Psi*}
\begin{split}
\Psi_*(t,\v) &= \frac{1}{4}(g'(t,\v))^2 - \frac{(n-2)^2}{4} g(t,\v)^2 - \kappa(\v) g(t,\v) \\
&\quad + \frac{n-2}{n} g(t,\v)^{\frac{2n-2}{n-2}} + 2\int_t^\infty  \left(\frac{1}{n\omega_n}\int_{\Ss^{n-1}} |\nabla_\theta\v|^2\,d\theta \right) g'(\tau,\v)\,d\tau\\
&\quad - \frac{2n-2}{n}\int_t^\infty \left(\frac{1}{n\omega_n} \int_{\Ss^{n-1}} |\v|^{\frac{2n}{n-2}}\,d\theta - g(\tau,\v)^{\frac{n}{n-2}} \right) g'(\tau,\v) \,d\tau,
\end{split}
\end{equation}
with $g'$ being $dg/dt$ and 
\begin{equation}\label{eq:g}
g(t,\v) = \frac{1}{n\omega_n} \int_{\Ss^{n-1}} |\v|^2\,d\theta. 
\end{equation}

\begin{proposition}\label{proposition:Phi*-const} Let $\u$ be a nonnegative solution of \eqref{eq:main} in $B_R\setminus\{0\}$ with $\alpha = \frac{n+2}{n-2}$, and let $\Phi_*(r,\u)$ be as in \eqref{eq:Phi*}. Then $\Phi_*(r,\u)$ is well-defined and is constant for $0<r<R$. 
\end{proposition}

We shall postpone the proof to Section \ref{section:asym-rad}, since proving the well-definedness of $\Phi_*(r,\u)$ essentially relies the asymptotic radial symmetry of local solutions to \eqref{eq:main} (see Theorem \ref{theorem:asym-rad}). 

\begin{remark}\label{remark:Phi*-const2} Knowing that $\Phi_*(r,\u)$ is constant, we obtain a constant $\kappa_*(\u)$ such that
\begin{equation}\label{eq:kappa*}
\kappa_*(\u) = \Phi_*(r,\u), 
\end{equation}
for any $0<r<R$. We shall call this constant the second Pohozaev invariant. As with the first Pohozaev invariant, we will also write it by $\kappa_*(\v)$ whenever $\v$ is the cylindrical transformation. Clearly,  
\begin{equation}\label{eq:kappa*-cyl}
\begin{split}
\kappa_*(\v) = \Psi_*(t,\v), 
\end{split}
\end{equation}
for any $t>-\log R$. In Section \ref{section:sing} and Section \ref{subsection:asym+sub} we will observe that $\kappa_*(\v) = 0$ if and only if $\v(t,\theta) = (1+o(1))|\v(t,\theta)|\ee$ uniformly for $\theta\in\Ss^{n-1}$ as $t\ra\infty$, with some unit nonnegative vector $\ee\in\R^m$. 
\end{remark}


\section{Solutions on the Whole Space}\label{section:global}

In this section we classify the smooth solutions of \eqref{eq:main} on the whole space $\R^n$. The analysis is based on the method of moving spheres along with the Kelvin transform, and we follow essentially the argument proposed by Li and Zhang \cite[Section 2]{LZ}, with only a minor modification. Nevertheless, we shall contain the full argument here for the reader's convenience. 

Given $z\in\R^n$ and $r>0$, we shall write $\u_{z,r}^*$ by the Kelvin transform of $\u$ with respect to the sphere $B_r(z)$, that is, 
\begin{equation}\label{eq:kelvin}
\u_{z,r}^*(y) = \left(\frac{r}{|y-z|}\right)^{n-2} \u \left( z + \frac{r^2}{|y-z|^2}(y-z) \right).
\end{equation}
Let us remark that if $\u$ is a solution of \eqref{eq:main} in $\R^n$, then
\begin{equation}\label{eq:kelvin-pde}
-\Delta \u_{z,r}^* = \left( \frac{r}{|y-z|} \right)^{(\alpha-1)\mu} |\u_{z,r}^*|^{\alpha-1} \u_{z,r}^*\quad\text{in }\R^n\setminus\{z\},
\end{equation}
where $\mu$ is given by \eqref{eq:lam-mu}. Note that $\mu\geq 0$ if and only if $1<\alpha\leq\frac{n+2}{n-2}$. The non-negativity of $\mu$ will play a key role when comparing $\u$ and $\u_{z,r}^*$. 

We  begin with a basic lemma that holds for any nonnegative, superharmonic function, as a starting point of the method of moving sphere. 

\begin{lemma}[Lemma 2.1 in \cite{LZ}]\label{lemma:basic} Let $v\in C^2(\R^n)$ be a super-harmonic and nonnegative function on $\R^n$. Then for each $z\in\R^n$, there exists $r_0>0$, which may depend on $v$ and $z$, such that for all $0<r<r_0$, 
\begin{equation}\label{eq:basic}
v_{z,r}^* \leq v\quad\text{in }\R^n\setminus B_r(z).
\end{equation}
\end{lemma}

The next lemma is an analogue of \cite[Lemma 2.4]{CGS} which claims that either the inequality \eqref{eq:basic} must hold until the solution becomes symmetric (with respect to a sphere) or it must fail on a compact subset of $\R^n$. The proof is given in that of \cite[Lemma 2.2]{LZ}, and we shall not repeat it here. 

\begin{lemma}\label{lemma:basic2} Let $v\in C^2(\R^n)$, $z\in\R^n$ and $r_0>0$ be such that 
\begin{equation}\label{eq:basic2-pde}
-\Delta (v - v_{z,r_0}^*) \geq 0 \quad\text{in }\R^n\setminus \bar{B}_{r_0}(z),
\end{equation}
and
\begin{equation}\label{eq:basic2-asmp}
v_{z,r_0}^* < v \quad\text{in }\R^n\setminus\bar{B}_{r_0}(z). 
\end{equation}
Then there is a small $\e>0$ such that for any $r_0<r<r_0 + \e$, 
\begin{equation}\label{eq:basic2}
v_{z,r}^* < v \quad\text{in }\R^n\setminus B_r(z). 
\end{equation} 
\end{lemma}

Now let us turn our interest to the nonnegative, smooth global solutions $\u$ of \eqref{eq:main}. Given $z\in\R^n$, let us define, for each $1\leq i\leq m$, 
\begin{equation}\label{eq:u22}
r_i(z) = \sup\{r>0:\text{$(u_i)_\rho^*\leq u_i$ in $\R^n\setminus B_\rho(z)$ for any $0<\rho<r$}\}.
\end{equation}
Since each component $u_i$ of $\u$ is nonnegative and superharmonic, Lemma \ref{lemma:basic} applies to $u_i$.  from which we know that $r_i(z) > 0$ for each $1\leq i\leq m$. Thus, we have 
\begin{equation}\label{eq:rb}
\bar{r}(z) = \inf_{1\leq i\leq m} r_i(z)>0.
\end{equation}

Let us remark that we have defined $\bar{r}(z)$ by the infimum, instead of minimum, of finite set of indices $\{1,2,\cdots,m\}$, since $r_i(z)$ as a supremum could be infinite. Moreover, if $r_i(z) = \infty$ for all $1\leq i\leq m$, we shall say that $\bar{r}(z) = \infty$. 

The following lemma takes care of the case when $\bar{r}(z)$ is either finite or infinite. The proof is essentially the same with \cite[Lemma 1.2, Lemma 1.3]{DHV}, which deals with the upper critical case $\alpha = \frac{n+2}{n-2}$ only, whence we shall skip the details. 

\begin{lemma}\label{lemma:u2} Let $\u$ be a nonnegative solution of \eqref{eq:main} in $\R^n$ with $1<\alpha\leq\frac{n+2}{n-2}$, $z\in\R^n$ be arbitrary and $\bar{r}(z)$ be as in \eqref{eq:rb}. If $\bar{r}(z)$ is finite, then  
\begin{equation}\label{eq:u2} 
\u_{z,\bar{r}(z)}^* = \u\quad\text{in }\R^n\setminus\{z\}. 
\end{equation}
If $\bar{r}(z_0) = \infty$ for some $z_0\in\R^n$, then $\bar{r}(z) = \infty$ for all $z\in\R^n$. 
\end{lemma}

We are now ready to classify the smooth global solutions.

\begin{proof}[Proof of Theorem \ref{theorem:main-g}] In view of Lemma \ref{lemma:u2}, we observe that $\bar{r}(z)$ defined in \eqref{eq:rb} is either finite or infinite for all $z\in\R^n$. If $\bar{r}(z)$ is finite for all $z\in\R^n$, then we have \eqref{eq:u2} at every point $z\in\R^n$. In this case, we may apply \cite[Lemma 11.1]{LZ} that there are $a_i\geq 0$, $r_i>0$ and $z_i\in\R^m$ for $1\leq i\leq m$ such that
\begin{equation}\label{eq:u-g1}
u_i(x) = a_ir_i^{-\frac{n-2}{2}}\left( \frac{r_i}{r_i^2 + |x-z_i|^2} \right)^{\frac{n-2}{2}}.
\end{equation}
On the other hand, if $\bar{r}(z)$ is infinite for all $z\in\R^n$, we have \eqref{eq:u22} for all $r>0$ at any $z\in\R^n$. Due to \cite[Lemma 11.2]{LZ}, there are $b_i\geq 0$ for $1\leq i\leq m$ such that
\begin{equation}\label{eq:u-g2}
u_i(x) = b_i.
\end{equation}

Suppose that $\u$ satisfies \eqref{eq:u-g2}, that is, $\u$ is constant everywhere on $\R^n$. As $\u$ being a nonnegative solution of \eqref{eq:main} in $\R^n$, $\u$ must be zero everywhere. Hence, Theorem \ref{theorem:main-g} (i) and (ii) are all satisfied under this assumption.

Next, let us consider the case that $u_i$ satisfies \eqref{eq:u-g1} for all $1\leq i\leq m$. This part is the same with the proof of \cite[Proposition 1.1]{DHV}, so we omit the details.
\end{proof}


\section{Solutions in Punctured Space}\label{section:sing}


\subsection{Radial Symmetry of Singular Solutions}\label{subsection:sing}

This section is devoted to the radial symmetry of nonnegative, singular solutions of \eqref{eq:main}. To be more precise, $\u$ is a nonnegative solution of \eqref{eq:main} in the punctured space $\R^n\setminus\{0\}$ that has a non-removable singularity at the origin, i.e., 
\begin{equation}\label{eq:sing}
\limsup_{x\ra 0} |\u(x)| = \infty.  
\end{equation}
The proof relies again on the method of moving spheres used in the previous section. The proof for the case of a single equation has already been established by Jin, et al. \cite[Proposition 2.1]{JLX}.  Nevertheless, the multiplicity in the components here makes the comparison argument more subtle, as observed in the previous section. Let us also address that the method of moving plane also works (c.f. \cite[Theorem 8.1]{CGS}) after a suitable modification. 

\begin{lemma}\label{lemma:rad} Let $\u$ be a nonnegative solution of \eqref{eq:main} in $\R^n\setminus\{0\}$ with $1<\alpha\leq\frac{n+2}{n-2}$. If $\u$ satisfies \eqref{eq:sing}, then $\u$ is radially symmetric. 
\end{lemma}

\begin{proof} Let $z\in\R^n\setminus\{0\}$ be arbitrary. Arguing similarly as with Lemma \ref{lemma:basic} (whose proof can be found in \cite[Lemma 2.1]{LZ}), there exists some $0<r_0<|z|$ such that for any $0<r\leq r_0$, 
\begin{equation*}
(u_i)_{z,r}^* \leq u_i\quad\text{in $\R^n\setminus (B_r(z)\cup\{0\})$ for each $1\leq i\leq m$}.
\end{equation*}
Hence, one can define, as with \eqref{eq:u22} and \eqref{eq:rb},
\begin{equation*}
r_i(z) = \sup\{r>0: (u_i)_{z,\rho}^* \leq u_i\text{ in $\R^n\setminus (B_\rho(z)\cup\{0\})$ for any $0<\rho<r$}\},
\end{equation*}
and
\begin{equation*}
\bar{r}(z) = \inf_{1\leq i\leq m} r_i(z).
\end{equation*}

We first claim that
\begin{equation}\label{eq:rbz-rad}
0<\bar{r}(z)\leq |z|.
\end{equation}
The positivity of $\bar{r}(z)$ is clear.To prove the second inequality in \eqref{eq:rbz-rad}, let us first observe that by \eqref{eq:sing}, there exist some sequence $x_j\ra 0$ and a component $u_i$ such that $u_i(x_j) \ra \infty$. If $\bar{r}(z)>|z|$, then by its definition, there should exist $\rho>|z|$ such that 
\begin{equation}\label{eq:rbz-rad-f}
(u_i)_{z,\rho}^* \leq u_i\quad\text{in }\R^n\setminus B_\rho(z).
\end{equation}
Now let $y_j$ be the reflection of $x_j$ with respect to $\partial B_\rho(z)$, i.e.,
\begin{equation*}
y_j  = z + \left( \frac{\rho}{|x_j - z|}\right)^2 (x_j-z).
\end{equation*}
Since $x_j\ra 0$, we have $y_j \in \R^n\setminus B_\rho(z)$ for all sufficiently large $j$, and moreover, 
\begin{equation*}
y_j \ra y_0 =  \left(1- \left( \frac{\rho}{|z|} \right)^2 \right) z. 
\end{equation*}
Thus, if we take $\rho$ close enough to $|z|$, we have $y_0\neq 0$, whence $u_i$ is smooth at $y_0$. However, \eqref{eq:rbz-rad-f} implies
\begin{equation*}
\begin{split}
u_i(y_0) &= \lim_{j\ra\infty} u_i(y_j) \geq \lim_{j\ra\infty} ((u_i)_{z,\rho}^* (y_j)) \geq \left(\frac{|z|}{\rho}\right)^{n-2} \lim_{j\ra\infty} u_i(x_j) = \infty,
\end{split}
\end{equation*}
a contradiction. 

From \eqref{eq:rbz-rad}, we can also claim that
\begin{equation*}
\bar{r}(z) = |z|.
\end{equation*}
The argument is based on the proof of \cite[Proposition 2.1]{JLX} with the corresponding modification shown in Lemma \ref{lemma:u2}, which amounts to the number of nontrivial components. The main idea is that if $\bar{r}(z) < |z|$, then \eqref{eq:sing} together with the maximum principle implies that 
\begin{equation}\label{eq:rbz-rad2-f}
u_i > (u_i)_{z,\bar{r}(z)}^*\quad\text{in }\R^n\setminus(\bar{B}_{\bar{r}(z)}(z)\cup\{0\}),
\end{equation}
at least for one $1\leq i\leq m$. Then we must have $|\u| > |\u_{z,\bar{r}(z)}^*|$ in $\R^n\setminus (\bar{B}_{\bar{r}(z)}(z)\cup\{0\})$, and the strong maximum principle yields that the strict inequality in \eqref{eq:rbz-rad2-f} must hold for all nontrivial components. Hence, as with Lemma \ref{lemma:basic2}, we obtain some $\e>0$ such that \eqref{eq:rbz-rad2-f} holds for all $1\leq i\leq m$ with $\bar{r}(z)$ replaced by some $\bar{r}(z) < r< \bar{r}(z) + \e$, a contradiction to \eqref{eq:rbz-rad-f}. The details are omitted. 

To this end, we have proved that for each $z\in\R^n\setminus\{0\}$ and for any $0<r<|z|$, 
\begin{equation*}
(u_i)_{z,r}^* \leq u_i \quad\text{in $\R^n\setminus (B_r(z)\cup\{0\})$ for each $1\leq i\leq m$}. 
\end{equation*}
Thus, one may deduce from \cite[Lemma 2.1]{JLX} that $u_i$ is radially symmetric for each $1\leq i\leq m$. 
\end{proof} 


\subsection{Limiting Energy Levels and Pohozaev Invariants}\label{subsection:kappa}

Knowing the radial symmetry of singular solutions, we may classify the nonnegative solutions on the punctured space, using the balanced-energy-limit. The idea is to consider both {\it blowups} and {\it shrink-downs} of $\u$ under the scaling \eqref{eq:ur}. Here by saying a blowup or a shrink-down under the scaling $\u_r$ we indicate a limit of $\u_r$ as $r = r_j\ra 0+$, or respectively $r = r_j\ra \infty$ in $C_{loc}^2(\R^n\setminus\{0\};\R^m)$. The following lemma provides the compactness of the sequence $\u_r$ in order to have both the blowups and the shrink-downs. 

\begin{lemma}\label{lemma:u-s} Let $\u$ be a nonnegative solution of \eqref{eq:main} in $\R^n\setminus\{0\}$ with $1<\alpha\leq\frac{n+2}{n-2}$. If $\u$ satisfies \eqref{eq:sing}, then for each $1\leq i\leq m$, 
\begin{equation}\label{eq:u-sup-s}
u_i(x) \leq \left(\frac{\alpha-1}{2n}\right)^{-\frac{1}{\alpha-1}}|x|^{-\frac{2}{\alpha-1}} \quad\text{in }\R^n\setminus\{0\}.
\end{equation}
\end{lemma}

\begin{proof} Let $u_i$ be a positive component of $\u$. Then, since $u_i$ is superharmonic in $\R^n\setminus\{0\}$, it follows from the extended maximum principle \cite[Theorem 1]{GilSer} that 
\begin{equation}\label{eq:u-liminf-s}
\liminf_{x\ra 0} u_i(x) > 0.
\end{equation}
Now let $v = u_i^{1-\alpha}$. Then $v$ satisfies, in $\R^n\setminus\{0\}$, 
\begin{equation*}
\Delta v \geq  \frac{\alpha}{\alpha-1} \frac{|\nabla v|^2}{v} + \alpha - 1.
\end{equation*}
Hence, for each $r>0$, the auxiliary function
\begin{equation*}
w(x) = v(x) - \frac{\alpha-1}{2n} |x|^2
\end{equation*}
becomes subharmonic in $B_r\setminus\{0\}$. Then by \eqref{eq:u-liminf-s}, $w$ is bounded around the origin, and thus, it follows from the extended maximum principle \cite[Theorem 1]{GilSer} that
\begin{equation*}
0 \leq \limsup_{x\ra 0} w(x) \leq \sup_{\partial B_r} w = \sup_{\partial B_r} v - \frac{\alpha-1}{2n} r^2. 
\end{equation*}
In terms of $u_i$, we obtain
\begin{equation*}
\inf_{\partial B_r} u_i \leq  \left(\frac{\alpha-1}{2n}\right)^{-\frac{1}{\alpha-1}}r^{-\frac{2}{\alpha-1}}.
\end{equation*}
Now the radial symmetry obtained in Lemma \ref{lemma:rad} yields \eqref{eq:u-sup-s}. 
\end{proof}

The next lemma gives the compactness of the sequence $\u_r$, and hence the existence of both blowup and shrink-down of $\u$. 

\begin{lemma}\label{lemma:u-cpt} Let $\u$ be a nonnegative solution of \eqref{eq:main} in $\R^n\setminus\{0\}$ with $1<\alpha\leq\frac{n+2}{n-2}$. Then there is some $0<\gamma<1$ such that $\u_r$ is uniformly bounded in $C^{2,\gamma}(K;\R^m)$ on each compact set $K\subset \R^n\setminus\{0\}$. 
\end{lemma}

\begin{proof} If $\u$ does not satisfy \eqref{eq:sing}, then $\u$ is bounded around the origin, and the origin becomes a removable singularity. According to Theorem \ref{theorem:main-g}, if $1<\alpha<\frac{n+2}{n-2}$, $\u$ is trivial, while if $\alpha = \frac{n+2}{n-2}$, $\u$ is globally bounded and satisfies $|\u(x)| = O(|x|^{2-n})$ as $|x|\ra\infty$. Hence, in any case, $\u_r$ is bounded uniformly for all $r>0$ on a fixed compact subset of $\R^n\setminus\{0\}$.

On the other hand, if $\u$ satisfies \eqref{eq:sing}, Lemma \ref{lemma:u-s} implies that $\u_r$ is globally bounded in $\R^n\setminus\{0\}$. Thus, regardless of the removability of the singularity at the origin, we know that $\u_r$ is uniformly bounded in each compact subset of $\R^n\setminus\{0\}$. 

Now since $\u_r$ also solves \eqref{eq:main} in $\R^n\setminus\{0\}$, it follows from the interior regularity theory \cite[Theorem 6.2 and Theorem 6.19]{GT} that $\u_r$ is uniformly bounded in $C^{2,\gamma}(K;\R^m)$ on each compact set $K\subset \R^n\setminus\{0\}$, for some $0<\gamma<1$. This finishes the proof. 
\end{proof}

Let $\Phi(r,\u)$ be the balanced-energy-type functional defined by \eqref{eq:Phi}. Recall from Proposition \ref{proposition:Phi-monot} that $\Phi(r,\u)$ is monotone increasing in $r>0$ for $1<\alpha<\frac{n+2}{n-2}$, while it is constant for $\alpha = \frac{n+2}{n-2}$. 

\begin{lemma}\label{lemma:u-Phi} Let $\u$ be a nonnegative solution of \eqref{eq:main} in $\R^n\setminus\{0\}$ with $1<\alpha\leq\frac{n+2}{n-2}$, and let $\u_0$ (resp., $\u_\infty$) be a blowup (resp., a shrink-down) under the scaling $\u_r$. Then $\Phi(r,\u_0) = \Phi(0+,\u)$ (resp., $\Phi(r,\u_\infty) = \Phi(\infty,\u)$) for all $r>0$. In particular, both $\u_0$ and $\u_\infty$ are homogeneous of degree $-\frac{2}{\alpha-1}$, provided that $1<\alpha<\frac{n+2}{n-2}$. 
\end{lemma}

\begin{proof} Since the argument for shrink-downs is the same, we shall only present it for blowups. Let $\u_0$ be a blowup with a sequence $r_j\ra 0+$. Then due to the scaling relation \eqref{eq:Phi-scale}, we have, for any $r>0$,
\begin{equation*}
\Phi(r,\u_0) = \lim_{j\ra\infty} \Phi(r,\u_{r_j}) = \lim_{j\ra\infty} \Phi(rr_j,\u) = \Phi(0+,\u),
\end{equation*}
where the existence of $\Phi(0+,\u)$ follows from the compactness of $\u_r$ (Lemma \ref{lemma:u-cpt}) and the monotonicity of $\Phi(r,\u)$ (Proposition \ref{proposition:Phi-monot} (i)). This proves the first assertion of Lemma \ref{lemma:u-Phi}. The second assertion on the homogeneity follows again from Proposition \ref{proposition:Phi-monot} (i). 
\end{proof}

\begin{lemma}\label{lemma:class-sub} Let $\u$ be a nonnegative solution of \eqref{eq:main} in $\R^n\setminus\{0\}$ with $1<\alpha \leq \frac{n+2}{n-2}$. Suppose further that $\u$ is homogeneous of degree $-\frac{2}{\alpha-1}$. 
\begin{enumerate}[(i)]
\item If $1<\alpha\leq\frac{n}{n-2}$, then $\u$ is trivial. 
\item If $\frac{n}{n-2} <\alpha \leq\frac{n +2}{n-2}$, then either $\u$ is trivial, or $\u$ is of the form \eqref{eq:u-s}.
\end{enumerate}
\end{lemma}

\begin{proof} Since $\u$ is homogeneous of degree $-\frac{2}{\alpha-1}$, the cylindrical transform $\v$ introduced in \eqref{eq:cyl} satisfies
\begin{equation}\label{eq:u0-pde-S}
\Delta_\theta \v- \lambda \v + |\v|^{\alpha-1}\v = 0\quad\text{on }\Ss^{n-1},
\end{equation}
where $\Delta_\theta$ is the Laplace-Beltrami operator, and $\lambda$ is given by \eqref{eq:lam-mu}. 

\begin{case} $1<\alpha\leq\frac{n}{n-2}$.
\end{case}
In view of \eqref{eq:lam-mu}, we have $\lambda\leq 0$. As a nonnegative solution of \eqref{eq:u0-pde-S}, we see that each component $v_i$ satisfies $\Delta_\theta v_i \leq 0$ on $\Ss^{n-1}$. This implies that $v_i$ does not attain any strict local minimum on $\Ss^{n-1}$. As $\Ss^{n-1}$ being a compact manifold, $v_i$ must be a constant. This argument holds for all $1\leq i\leq m$, which makes $\v$ a nonnegative, constant vector on $\Ss^{n-1}$. However, a nonnegative constant solution of \eqref{eq:u0-pde-S} must be trivial because $\lambda\leq 0$. Returning back to $\u$, it indicates that $\u$ is trivial on $\partial B_1$. As each of its component being nonnegative and superharmonic, $\u$ must be trivial in the whole domain, which proves Lemma \ref{lemma:class-sub} (i).

\begin{case} $\frac{n}{n-2}<\alpha<\frac{n+2}{n-2}$. 
\end{case}

Suppose that $\u$ is a nontrivial solution in the punctured space. Then by the non-negativity and the super-harmonicity of each component of $\u$, $|\u|$ is positive everywhere. As is homogeneous of degree $-\frac{2}{\alpha-1}$, $\u$ must have a non-removable singularity at the origin, i.e., \eqref{eq:sing} holds. By Lemma \ref{lemma:rad}, $\u$ is radially symmetric, whence $\u$ is a positive constant vector, $\a$, on $\partial B_1$. 

By \eqref{eq:u0-pde-S} we have $|\a| = \lambda^{\frac{1}{\alpha-1}}$. By the homogeneity, we see that $\u$ is of the form $\lambda^{\frac{1}{\alpha-1}} |x|^{-\frac{2}{\alpha-1}} \ee$ with some nonnegative unit $\ee\in\R^m$, proving Lemma \ref{lemma:class-sub} (ii). 
\end{proof}

We are in a position to prove Theorem \ref{theorem:main-s} (i) and (ii). 

\begin{proof}[Proof of Theorem \ref{theorem:main-s} (i) and (ii)] Let $\u_0$ and $\u_\infty$ be a blowup and, respectively, a shrink-down of $\u$. According to Lemma \ref{lemma:u-Phi}, both $\u_0$ and $\u_\infty$ are homogeneous of degree $-\frac{2}{\alpha-1}$. Hence, it follows from Lemma \ref{lemma:class-sub} (i) that if $1<\alpha\leq\frac{n}{n-2}$, both $\u_0$ and $\u_\infty$ are trivial. This in turn yields by Lemma \ref{lemma:u-Phi} that $\Phi(0+,\u) = \Phi(\infty,\u) = 0$. Due to the monotonicity of $\Phi(r,\u)$, $\Phi(r,\u) = 0$ for all $r>0$. Thus, by Proposition \ref{proposition:Phi-monot} (i), $\u$ is homogeneous of degree $-\frac{2}{\alpha-1}$. Theorem \ref{theorem:main-s} (i) is now an immediate consequence of Lemma \ref{lemma:class-sub} (i). 

Now let us consider the case $\frac{n}{n-2}<\alpha<\frac{n+2}{n-2}$. By Lemma \ref{lemma:u-Phi} and Lemma \ref{lemma:class-sub} (ii), any blowup $\u_0$ is either trivial or of the form \eqref{eq:u-s}. If $\u_0$ is trivial, then clearly $\Phi(r,\u_0) = 0$ for all $r>0$, which along with Lemma \ref{lemma:u-Phi} implies that $\Phi(0+,\u) = 0$. On the other hand, if $\u_0$ is of the form \eqref{eq:u-s}, then a simple computation shows that $\Phi(r,\u_0) = -\bar\lambda$ for all $r>0$, with $\bar\lambda$ given as in \eqref{eq:lamb}. Thus, again from Lemma \ref{lemma:u-Phi} it follows that $\Phi(0+,\u) = - \bar\lambda$. The converse statement is obviously true, whence we have proved that $\Phi(0+,\u) \in \{-\bar\lambda,0\}$, and $\Phi(0+,\u) = 0$ if and only if all the blowups are trivial, while $\Phi(0+,\u) = -\bar\lambda$ if and only if all the blowups are of the form \eqref{eq:u-s}. 

Further, the same assertion holds for any shrink-down $\u_\infty$, proving that $\Phi(\infty,\u) \in \{-\bar\lambda,0\}$, and $\Phi(\infty,\u) = 0$ if and only if all the shrink-downs are trivial, while $\Phi(\infty,\u) = -\bar\lambda$ if and only if all the shrink-downs are of the form \eqref{eq:u-s}. 

Now if $\Phi(0+,\u) = 0$, then since $\Phi(r,\u)$ is nondecreasing in $r$ and $\Phi(\infty,\u)\in \{-\bar\lambda,0\}$, we must have $\Phi(r,\u) = 0$ for all $r>0$. Hence, by Lemma \ref{lemma:u-Phi} and Lemma \ref{lemma:class-sub} (ii), $\u$ is either trivial or of the form \eqref{eq:u-s}. However, the latter yields that $\Phi(0+,\u) = -\bar\lambda$, a contradiction. Thus, $\u$ must be trivial. Of course, the converse is also true. 

Similarly, $\Phi(\infty,\u) = -\bar\lambda$ implies that $\u$ is of the form \eqref{eq:u-s}. This finishes the proof of Theorem \ref{theorem:main-s} (ii). 
\end{proof}

The analysis on the case $\alpha = \frac{n+2}{n-2}$ is more subtle. Our approach relies on the Pohozaev invariants of which the first one  $\kappa(\u)$ was introduced in \eqref{eq:kappa}.  In the following we focus on the second Pohozaev invariant $\kappa_*(\u)$,  which was briefly introduced in Remark \ref{remark:Phi*-const2}. More importantly, we shall observe that this second invariant appears solely due to the multiplicity of the components of \eqref{eq:main}. 

\begin{lemma}\label{lemma:class-crit} Let $\u$ be a nonnegative solution of \eqref{eq:main} in $\R^n\setminus\{0\}$ with $\alpha = \frac{n+2}{n-2}$. Then $\Phi(r,\u)$ and $\Phi_*(r,\u)$ in \eqref{eq:Phi} and \eqref{eq:Phi*} are well-defined, and there are constants $\kappa(\u)$ and $\kappa_*(\u)$ satisfying \eqref{eq:kappa} and respectively \eqref{eq:kappa*}. Moreover, the inequalities \eqref{eq:class-crit-kappa} and \eqref{eq:class-crit-kappa*} holds and the equalities of the lower bounds only occur simultaneously. 
\end{lemma}

\begin{proof}
The proof can be divided into two cases; first we consider the case where $\u$ is not rotationally symmetric, and then we treat the other case. We shall prove the equivalent statements for the cylindrical transformation $\v$. Since $\v$ will be fixed throughout the proof, we shall omit the dependence of $\Psi$, $\Psi_*$, $\kappa$ and $\kappa_*$ on $\v$ here. 

Suppose that $\u$ is not rotationally symmetric. Due to Lemma \ref{lemma:rad}, $\u$ has a removable singularity at the origin. Thus, its cylindrical transformation $\v$, given as in \eqref{eq:cyl}, satisfies 
\begin{equation}\label{eq:class-crit-1}
|\v(t,\theta)| + |\partial_t\v(t,\theta)| \leq Ce^{-\frac{n-2}{2}t}\quad\text{on }\Ss^{n-1}, 
\end{equation}
as $t\ra\infty$, with some constant $C>0$ independent of $t$. This combined with \eqref{eq:kappa-cyl} implies that 
\begin{equation}\label{eq:class-crit-2}
\kappa = \lim_{t\ra\infty} \Psi(t) = 0.
\end{equation} 

On the other hand, the estimate \eqref{eq:class-crit-1} also ensures the well-definedness of $\Psi_*(t)$ given by \eqref{eq:Psi*} for all $t\in\R$. To prove that $\Psi_*(t)$ is constant for any $t\in\R$, we need to compute the derivatives of $g$, given by \eqref{eq:g}. Utilizing \eqref{eq:main-cyl}, \eqref{eq:kappa-cyl} and \eqref{eq:class-crit-2} one can verify that 
\begin{equation*}
g'' = \frac{2}{n\omega_n} \int_{\Ss^{n-1}} \left( \frac{(n-2)^2}{2} |\v|^2 + 2|\nabla_\theta\v|^2 - \frac{2n-2}{n} |\v|^{\frac{2n}{n-2}} \right) \,d\theta,
\end{equation*}
from which it follows that
\begin{equation}\label{eq:Psi*'}
\begin{split}
&\Psi_*'(t) \\
&= g' \left( \frac{g''}{2} - \frac{(n-2)^2}{2} g  - \frac{2}{n\omega_n} \int_{\Ss^{n-1}} \left( |\nabla_\theta \v|^2 - \frac{n-1}{n} |\v|^{\frac{2n}{n-2}} \right)\,d\theta\right)  \\
&= \frac{g'}{n\omega_n}\int_{\Ss^{n-1}} \left( |\partial_t \v|^2 -\frac{(n-2)^2}{4} |\v|^2  - |\nabla_\theta \v|^2 + \frac{n-2}{n} |\v|^{\frac{2n}{n-2}} \right) \,d\theta\\
&= 0.
\end{split}
\end{equation}
Thus, $\Psi_*(t)$ is constant for any $t\in\R$, and there must exist a constant $\kappa_*(\v)$ such that \eqref{eq:kappa*-cyl} holds for all $t$. Moreover, one can also verify from \eqref{eq:class-crit-1} that 
\begin{equation*}
\kappa_* =  \lim_{t\ra\infty} \Psi_*(t) = 0.
\end{equation*}
This proves the lemma for the case where $\u$ is rotationally symmetric. 

Next we consider the case $\u$ is rotationally symmetric, so that the cylindrical transformation $\v$ becomes a function of $t$ only. In this case, we have already observed that \eqref{eq:|v|-poho} holds with $\kappa_*$ given by \eqref{eq:2nd-poho}. Note that under the rotational symmetry of $\v$, $g$ as in \eqref{eq:g} is identical to $|\v|^2$. Hence, one can easily observe from \eqref{eq:Psi*} and \eqref{eq:|v|-poho} that
\begin{equation}\label{eq:kappa*-cyl-rad}
\Psi_*(t) = \frac{(g')^2}{4} - \frac{(n-2)^2}{4} g^2 - \kappa g + \frac{n-2}{n} g^{\frac{2n-2}{n-2}} = \kappa_*,
\end{equation} 
as desired. 

Let us now prove the bounds in \eqref{eq:class-crit-kappa} and \eqref{eq:class-crit-kappa*}. Since we have already verified above that $\kappa = \kappa_* = 0$ if $\v$ is not rotationally symmetric, it suffices to consider the situation where $\v$ is rotationally symmetric. Then one can follow the derivation of \eqref{eq:2nd-poho} and verify that $\kappa_* \leq 0$. Hence, we are only left with proving the lower bounds of $\kappa$ and $\kappa_*$. 

Set 
\begin{equation*}
f(s) = \frac{(n-2)^2}{4} s^2 - \frac{n-2}{n} s^{\frac{2n-2}{n-2}} + \kappa s,
\end{equation*}
and  let us rephrase the second identity in \eqref{eq:kappa*-cyl-rad} as
\begin{equation}\label{eq:kappa*-cyl-rad-re}
\frac{(g')^2}{4} = f(g) + \kappa_*.
\end{equation}
Utilizing $\kappa_* \leq 0$ in the identity above, we see that $f(g) \geq 0$. Since either $g(t)=0$ and $g(t)>0$ for all $t$, and $g(t)=0$ yields $\kappa=0$, we can focus on the case $g(t)>0$ for all $t$. Then $\frac{1}{g} f(g) \geq 0$ as well, from which it follows that 
\begin{equation*}
\kappa \geq - \frac{(n-2)^2}{4} g + \frac{n-2}{n} g^{\frac{n}{n-2}} \geq -\frac{2}{n} \left( \frac{n-2}{2} \right)^n.
\end{equation*}
This verifies the lower bound \eqref{eq:class-crit-kappa} of $\kappa$. 

To verify the the lower bound \eqref{eq:class-crit-kappa*} of $\kappa_*$, let us remark that 
\begin{equation*}
\left( \frac{2}{n} \left(\frac{n-2}{2}\right)^n + \kappa \right)\left(\frac{n-2}{2}\right)^{n-2} = f \left( \left(\frac{n-2}{2}\right)^{n-2} \right). 
\end{equation*}
Now suppose towards a contradiction that there is a solution $\v$ having $\kappa_* < -f ( (\frac{n-2}{2})^{n-2})$. Then it follows from \eqref{eq:kappa*-cyl-rad-re} that $\min\{g (t) : t\in\R\} > (\frac{n-2}{2})^{n-2}$, or equivalently, $\min\{ |\v(t)| : t\in\R\} > (\frac{n-2}{2})^{\frac{n-2}{2}}$. In view of \eqref{eq:main-cyl}, this implies that 
\begin{equation}\label{eq:class-crit-kappa*-1}
v_i'' = \frac{(n-2)^2}{4} v_i - |\v|^{\frac{4}{n-2}} v_i \leq - \delta v_i,
\end{equation}
for each $1\leq i\leq m$, where $\delta = \min\{|\v(t)| : t\in\R\} - (\frac{n-2}{2})^{\frac{n-2}{2}} > 0$. Hence, $v_i$ is a concave function. However, \eqref{eq:u-sup-s} shows that $v_i$ is uniformly bounded for all $t$, which indicates that $v_i(t) \ra a_i$ and $v_i''(t) \ra 0$ as $t\ra\infty$ for some $a_i>0$. However, this is a contradiction against \eqref{eq:class-crit-kappa*-1}, which proves the lower bound \eqref{eq:class-crit-kappa*} of $\kappa_*$. 

Finally, let us investigate the scenario when the equalities of the lower bounds in \eqref{eq:class-crit-kappa} and \eqref{eq:class-crit-kappa*} hold. Suppose that the equality of the lower bound in \eqref{eq:class-crit-kappa*} occur. That is, 
\begin{equation}\label{eq:class-crit-kappa*-2}
\kappa + \left(\frac{n-2}{2}\right)^{2-n} \kappa_* = - \frac{2}{n} \left(\frac{n-2}{2}\right)^n.
\end{equation}
Arguing similarly as above, one can deduce that $\min\{|\v(t)|:t\in\R\} \geq (\frac{n-2}{2})^{\frac{n-2}{2}}$ and $v_i'' \leq 0$ in $\R$ for each $1\leq i\leq m$. Again $v_i$ is a concave function that is uniformly bounded in $\R$, so $v_i(t) \ra a_i$, for some $a_i\in\R$, and $v_i''(t) \ra 0$ as $t\ra\infty$. Thus, $|\v(t)|\ra |\a|$ with $\a = (a_1,\cdots,a_m)$, and it follows from $v_i''(t) \ra 0$ and the first equality in \eqref{eq:class-crit-kappa*-1} that $|\a| = (\frac{n-2}{2})^{\frac{n-2}{2}}$. On the other hand, we also have $v_i'(t)\ra 0$ as $t\ra\infty$, so sending $t\ra\infty$ in the second equality of \eqref{eq:kappa-cyl} yields that 
\begin{equation*}
\kappa = \lim_{t\ra\infty} \left( |\v'(t)|^2 - \frac{(n-2)^2}{4} |\v(t)|^2 + \frac{n-2}{n} |\v(t)|^{\frac{2n}{n-2}} \right) = -\frac{2}{n} \left(\frac{n-2}{2}\right)^n.
\end{equation*}
Thus, \eqref{eq:class-crit-kappa*-2} forces $\kappa_* = 0$, and the final assertion of the lemma is proved. 
\end{proof} 

Let us finish this section by proving Theorem \ref{theorem:main-s} (iii).

\begin{proof}[Proof of Theorem \ref{theorem:main-s} (iii)] 
The well-definedness and the bounds of $\kappa$ and $\kappa_*$ are proved in Lemma \ref{lemma:class-crit}. The other assertions can be proved as follows. 

First consider the assertion (iii)-(a). If $\u$ is not radially symmetric, then by Lemma \ref{lemma:rad}, $\u$ has a removable singularity at the origin, as desired. On the other hand, if $\u$ is radially symmetric, one can deduce from \eqref{eq:|v|-poho} that the cylindrical transformation $\v$, which is now a function of $t$ only, satisfies 
\begin{equation}\label{eq:kappa-cyl-rad-reg}
\left(\frac{d|\v|}{dt} \right)^2 = \frac{(n-2)^2}{4} |\v|^2 - \frac{n-2}{n} |\v|^{\frac{2n}{n-2}}.
\end{equation}
Hence, the classical work such as \cite{F} or \cite{CGS} applies to $|\v|$, proving the `only if' part of the assertion (iii)-(a). The `if' part can be verified through a direct computation.

Let us move on to the case $\kappa^2 + \kappa_*^2 > 0$. From the assertion (iii)-(a), we see that $\u$ must have a non-removable singularity at the origin. According to Lemma \ref{lemma:rad}, $\u$ is radially symmetric, so one can follow the computation in Section \ref{section:intro} and deduce \eqref{eq:|v|-poho}. 

Finally, assume that $\kappa = -\frac{2}{n} (\frac{n-2}{2})^n$ and $\kappa_* = 0$. It follows from \eqref{eq:|v|-poho} that 
\begin{equation*}
\left(\frac{d|\v|}{dt}\right)^2 - \frac{(n-2)^2}{4}|\v|^2 + \frac{n-2}{n} |\v|^{\frac{2n}{n-2}} + \frac{2}{n} \left(\frac{n-2}{2}\right)^n = 0,
\end{equation*}
whence $|\v|$ has to be constant in $\R$, and the constant has to be $(\frac{n-2}{2})^{\frac{n-2}{2}}$. In terms of $\u$ this implies that $\u$ is homogeneous of degree $-\frac{n-2}{2}$ and is of the form \eqref{eq:u-s-2}. This constitutes the `only if' part of the assertion (iii)-(c). The `if' part follows easily from a direct computation. 
\end{proof}


\section{A Priori Estimate and Harnack-Type Inequality for Local Solutions}\label{section:upper}

In this section, we prove {\it a priori} upper bounds for local solutions of \eqref{eq:main} in $B_1\setminus\{0\}$ with $1<\alpha \leq \frac{n+2}{n-2}$ which further allows us to derive related Harnack inequalities, interior gradient estimates and the compactness of scaling functions. Our analysis is divided into two cases, according to the subcritical range $1<\alpha<\frac{n+2}{n-2}$ and the critical range $\alpha = \frac{n+2}{n-2}$. The former is based on the non-existence of the smooth, positive, global solution in Theorem \ref{theorem:main-g} (i) along with a blowup argument. The latter uses the method of moving spheres presented in the previous section, essentially following the work of Li and Zhang \cite{LZ}. 


\subsection{A Priori Bound for $1<\alpha<\frac{n+2}{n-2}$}\label{subsection:upper-sub}

We begin with the upper bound for the subcritical case, which is (much) simpler than the critical case. 

\begin{proposition}\label{proposition:sup-sub} Let $1<\alpha<\frac{n+2}{n-2}$ and suppose that $\v\in C^2(B_1;\R^m)\cap C(\bar{B}_1;\R^m)$ is a nonnegative solution of
\begin{equation}\label{eq:main-reg}
-\Delta \v = |\v|^{\alpha-1}\v \quad\text{in }B_1.
\end{equation}
Then there exists $C>0$, depending only on $n$, $m$ and $\alpha$, such that
\begin{equation}\label{eq:sup}
|\v(x)| \leq C(1-|x|)^{-\frac{2}{\alpha-1}} \quad\text{in }B_1. 
\end{equation}
\end{proposition}

\begin{proof} 
Note that $w = v_1 + \cdots + v_m$ satisfies 
\begin{equation*}
\frac{1}{c} w^\alpha \leq -\Delta w \leq c w^\alpha,
\end{equation*} 
for some $c>1$, depending only on $m$ and $\alpha$. Thus, we can follow the proof of \cite[Theorem 2.1]{PQS} and obtain the desired inequality. We omit the details. 
\end{proof}


\subsection{A Harnack-Type Inequality for $\alpha=\frac{n+2}{n-2}$}\label{subsection:upper-crit}

Our approach to achieve the Harnack-type inequality for $\alpha=\frac{n+2}{n-2}$ follows the line of the scalar case in Li and Zhang \cite[Lemma 5.1]{LZ}. In our system setting, the problem becomes very sensitive on the number of nonzero components, and we modify the proof of \cite[Lemma 5.1]{LZ} in this direction. 

\begin{proposition}\label{proposition:har-crit} Let $\v\in C^2(B_2;\R^m)\cap C(\bar{B}_2;\R^m)$ be a nonnegative solution of
\begin{equation}\label{eq:main-crit}
-\Delta \v = |\v|^{\frac{4}{n-2}}\v\quad\text{in }B_2.
\end{equation}
Then, there exists $C>0$ depending only on $n$ and $m$, such that
\begin{equation}\label{eq:har-crit}
\left( \min_{i \in I_m} \inf_{\partial B_2} v_i\right) |\v(x)| \leq C (1-|x|)^{-\frac{n-2}{2}} \quad\text{in }B_1,
\end{equation}
where $I_m$ is the set of indices $1\leq i\leq m$ such that $v_i$ is nontrivial. 
\end{proposition}

\begin{proof} If $\v$ is trivial, then $I_m = \emptyset$, whence there is nothing to prove. Thus, we shall assume that $\v$ is not trivial, so that $I_m \neq\emptyset$. Then for each $i\in I_m$,  we know from the super-harmonicity and the non-negativity of $v_i$ that $\inf_{\partial B_2} v_i > 0$, whence $(\min_{i\in I_m} \inf_{\partial B_2} v_i)^{-1}$ is a positive, finite number. 

If $|\v(x)| \leq C_1(1 - |x|)^{-\frac{n-2}{2}} $ in $B_1$ for some $C_1>0$ depending only on $n$ and $m$, then the claim \eqref{eq:har-crit} is true, since the maximum principle and the super-harmonicity of each component of $\v$ implies that $\inf_{\partial B_2} v_i  \leq v_i(0)$. Thus, let us assume that for all $j\geq 1$ there are nonnegative solutions $\v_j$ of \eqref{eq:main-crit} and points $x_j\in \bar{B}_1$ such that 
\begin{equation}\label{eq:sup-f}
M_j:= \sup_{|x|\leq 1} \left((1-|x|)^{\frac{n-2}{2}} |\v_j(x)|\right) =(1-|x_j|)^{\frac{n-2}{2}} |\v_j(x_j)| \to \infty.
\end{equation}
We know that $x_j\in B_1$ (instead of $\partial B_1$) since $\v_j$ is continuous on $\bar{B}_1$. Moreover, we shall set 
\begin{align}\label{eq:rj-f}
r_j &= \frac{1}{2}(1-|x_j|)>0,\\
\label{eq:dj}
\delta_j &= |\v_j(x_j)|^{-\frac{\alpha-1}{2}} = 2r_j M_j^{-\frac{2}{n-2}} \ra 0,\\
\label{eq:Rj}
R_j &= \frac{r_j}{\delta_j} = \frac{1}{2} M_j^{\frac{2}{n-2}} \ra \infty.
\end{align} 
It should be noted that due to \eqref{eq:sup-f}, we have 
\begin{equation}\label{eq:uj-xj}
|\v_j(x)| \leq \left( \frac{1 - |x_j|}{1-|x|}\right)^{\frac{2}{\alpha-1}} |\v_j(x_j)| \leq 2^{\frac{2}{\alpha-1}} |\v_j(x_j)|\quad\text{in }B_{r_j}(x_j).
\end{equation}
In addition, inserting \eqref{eq:rj-f} into \eqref{eq:sup-f}, we obtain
\begin{equation}\label{eq:uj-xj2}
|\v_j(x_j)| = (2r_j)^{-\frac{2}{\alpha-1}} M_j.
\end{equation}

With \eqref{eq:uj-xj} and \eqref{eq:uj-xj2} at hand, one can following the proof of \cite[Lemma 5.1]{LZ} to deduce that the sequence of the scaled function, 
\begin{equation*}
\w_j(x) = \delta_j^{\frac{n-2}{2}} \v_j( \delta_j x + x_j)\quad\text{in }B_{R_j},
\end{equation*}
converges to $w_0$ in $C_{loc}^2(\R^n;\R^m)$ for certain $\w_0\in C^2(\R^n;\R^m)$, which is a nonnegative solution of 
\begin{equation}\label{eq:ut0-pde}
-\Delta\w_0 = |\w_0|^{\frac{4}{n-2}}\w_0\quad\text{in }\R^n,
\end{equation}
satisfying 
\begin{equation}\label{eq:ut0-sup}
|\w_0(x)|\leq 2^{\frac{2}{\alpha-1}} \quad\text{in }\R^n,
\end{equation}
as well as 
\begin{equation}\label{eq:ut0-0}
|\w_0(0)| = 1. 
\end{equation}
We omit the details here. 

With only a minor modification, one may apply Lemma \ref{lemma:basic} to each component $w_{i,j}$ of $\w_j$, with $i\in I_m$, and obtain a number $s_{i,j}(z)>0$, corresponding to each $z\in\R^n$, such that for all $0<r<s_{i,j}(z)$, 
\begin{equation}\label{eq:wjk-basic}
(w_{i,j})_{z,r}^* \leq w_{i,j}\quad\text{in } B_{1/(2\delta_j)}(z) \setminus B_r(z). 
\end{equation} 
Here we choose $j$ large enough so that $B_{1/(2\delta_j)}(z)\subset B_{1/\delta_j}$, which is possible due to \eqref{eq:dj}. One may refer to the proof of \cite[Theorem 1.5]{LZ} for the details.  

Let us now replace $s_{i,j}(z)$ by the supremum value of $r$ such that \eqref{eq:wjk-basic} holds, that is, 
\begin{equation}\label{eq:sijz}
s_{i,j}(z) = \sup\{ r: (w_{i,j})_{z,\rho}^* \leq w_{i,j}\text{ in $B_{1/(2\delta_j)}(z) \setminus B_r(z)$ for any $0<\rho<r$}\}.
\end{equation}
Now with $s_{i,j}(z)$ defined as in \eqref{eq:sijz}, we shall set, analogously to \eqref{eq:rb}, 
\begin{equation}\label{eq:rbj}
\bar{s}_j(z) = \inf_{i \in I_m} s_{i,j}(z).
\end{equation}
Then we have
\begin{equation}
(w_{i,j})_{z,\bar{s}_j(z)}^* \leq w_{i,j}\quad\text{in $B_{1/(2\delta_j)}(z) \setminus B_{\bar{s}_j(z)}(z)$ for each $i \in I_m$},
\end{equation}
and respectively,
\begin{equation}\label{eq:wij-kelvin-pde}
-\Delta (w_{i,j} - (w_{i,j})_{z,\bar{s}_j(z)}^*) \geq 0 \quad\text{in }B_{1/(2\delta_j)}(z) \setminus \bar{B}_{\bar{s}_j(z)}(z).
\end{equation}

Now let us assume towards a contradiction that
\begin{equation}\label{eq:vj-f}
\min_{i\in I_m}\inf_{\partial B_2} v_{i,j} \geq j \left( \sup_{|x|\leq 1} (1-|x|)^{\frac{n-2}{2}} |\v_j(x)| \right)^{-1} = \frac{j}{M_j}.
\end{equation}
In terms of $w_{i,j}$, one may rewrite \eqref{eq:vj-f} as 
\begin{equation}\label{eq:wj-f}
\begin{split}
\min_{i\in I_m} \inf_{\partial B_{1/\delta_j}} w_{i,j} &= \delta_j^{\frac{n-2}{2}} \min_{i\in I_m} \inf_{\partial B_1(x_j)} v_{i,j} \\
&\geq \delta_j^{\frac{n-2}{2}} \min_{i\in I_m} \inf_{\partial B_2} v_{i,j}\\
&\geq j \delta_j^{n-2},
\end{split}
\end{equation}
where in the derivation of the first inequality we used the super-harmonicity of $v_{i,j}$, the maximum principle and the fact that $B_1(x_j)\subset B_2$, while the second inequality follows from \eqref{eq:vj-f}, \eqref{eq:dj} and the fact that $2r_j = 1-|x_j| \leq 1$. 

In view of \eqref{eq:wj-f}, one may easily deduce that for any $z\in\R^n$, 
\begin{equation}\label{eq:rbj-f}
\lim_{j\ra\infty} \bar{s}_j(z) = \infty.
\end{equation}
Suppose that \eqref{eq:rbj-f} is false, and there exists some $L>0$, independent of $j$, such that
\begin{equation}\label{eq:rbj-ff}
\bar{s}_j(z) \leq L. 
\end{equation}
Then by the definition of the Kelvin transform (see \eqref{eq:kelvin}), we have, for any $i\in I_m$, 
\begin{equation}\label{eq:wj-ff}
\begin{split}
\sup_{\partial B_{1/(4\delta_j)}(z)} (w_{i,j})_{z,\bar{s}_j(z)}^* &= (4\delta_j\bar{s}_j(z))^{n-2} \sup_{\partial B_{4\delta_j\bar{s}_j^2(z)}} w_{i,j} \\
&\leq (4\delta_j L)^{n-2} \delta_j^{\frac{n-2}{2}} \sup_{B_{4\delta_j^2L^2}} v_{i,j} \\
&\leq (8L)^{n-2} \delta_j^{n-2},
\end{split}
\end{equation}
where in deriving the first and the second inequality we used \eqref{eq:rbj-ff} and, respectively, \eqref{eq:uj-xj} with \eqref{eq:uj-xj2}. According to \eqref{eq:wj-f} and \eqref{eq:wj-ff}, for each $i\in I_m$,  
\begin{equation}\label{eq:wj-fff}
\inf_{\partial B_{1/(4\delta_j)}(z)} (w_{i,j} - (w_{i,j})_{z,\bar{s}_j(z)}^*) \geq (j - (8L)^{n-2}) \delta_j^{n-2} > 0,
\end{equation}
for all sufficiently large $j$, where in the first inequality we used $w_{i,j}\geq \inf_{\partial B_{1/\delta_j}} w_{i,j}$ on $\partial B_{1/(4\delta_j)}(z)$, which follows from the maximum principle, the super-harmonicity of $w_{i,j}$ in $B_{1/\delta_j}$ and the fact that $B_{1/(4\delta_j)(z)} \subset B_{1/\delta_j}$. With \eqref{eq:wj-fff} at hand, we may apply the maximum principle to \eqref{eq:wij-kelvin-pde} and observe that for any $i\in I_m$, 
\begin{equation}\label{eq:wj-f4}
(w_{i,j})_{z,\bar{s}_j(z)}^* < w_{i,j} \quad\text{in }B_{1/(2\delta_j)(z)}\setminus \bar{B}_{\bar{s}_j(z)}(z). 
\end{equation}

Now that $w_{i,j}$ satisfies \eqref{eq:wij-kelvin-pde} and \eqref{eq:wj-f4} for each $i\in I_m$, we can follow a similar argument to that in the proof of \cite[Lemma 5.2]{LZ} and deduce that there exist $\bar{s}_{i,j}(z) > \bar{s}_j(z)$ and $0<\e_{i,j} < \bar{s}_{i,j}(z) - \bar{s}_j(z)$ such that for any $\bar{s}_j(z)< r < \bar{s}_j(z) + \e_{i,j}$, 
\begin{equation}\label{eq:wj-f5}
(w_{i,j})_{z,\bar{s}_j(z)}^* < w_{i,j}\quad\text{in $B_{1/(2\delta_j)(z)}\setminus B_r(z)$ for each $i\in I_m$}.
\end{equation}
Clearly, \eqref{eq:wj-f5} violates the definition of $\bar{s}_j(z)$ in \eqref{eq:rbj}. Hence, the claim \eqref{eq:rbj-f} should be true, under the assumption \eqref{eq:vj-f}.

Knowing that \eqref{eq:wj-f} is true for all $z\in\R^n$ (under the assumption \eqref{eq:vj-f}), we have for any $z\in\R^n$ and $r>0$ that 
\begin{equation}\label{eq:wij-fff}
(w_{i,j})_{z,r}^* \leq w_{i,j}\quad\text{in $B_{1/(2\delta_j)}(z)\setminus B_r(z)$ for any $i\in I_m$},
\end{equation}
for all sufficiently large $j$ such that $\bar{s}_j(z) > r$. On the other hand, recall from the beginning of this proof that $\w_j\ra\w_0$ in $C_{loc}^2(\R^n;\R^m)$ with some $\w_0\in C^2(\R^n;\R^m)$ satisfying \eqref{eq:ut0-pde}, \eqref{eq:ut0-sup} and \eqref{eq:ut0-0} with $\alpha = \frac{n+2}{n-2}$. This implies $(\w_j)_{z,r}^* \ra (\w_0)_{z,r}^*$ in $C_{loc}^2(\R^n\setminus \{z\};\R^m)$ for each $z\in\R^n$ and any $r>0$. Thus, we may pass to the limit with $j\ra\infty$ (possibly along a subsequence) in \eqref{eq:wij-fff} in any compact domain of type $B_R(z)\setminus B_r(z)\subset \R^n\setminus\{z\}$, which gives 
\begin{equation}\label{eq:wi0-fff}
(w_{i,0})_{z,r}^* \leq w_{i,0} \quad\text{in $\R^n\setminus B_r(z)$ for any $i\in I_m$}.
\end{equation}

As $z\in\R^n$ and $r>0$ in \eqref{eq:wi0-fff} being arbitrary, we conclude from \cite[Lemma 11.2]{LZ} that $w_{i,0}$ is constant for each $i\in I_m$. Then as $w_{i,0}$ being a nonnegative (global) solution of \eqref{eq:ut0-pde}, $w_{i,0}$ must be trivial for each $i\in I_m$. On the other hand, for any $i\not\in I_m$, $v_i$ is already trivial and so is the limit $w_{i,0}$. Consequently, $\w_0$ is a trivial solution, a contradiction against \eqref{eq:ut0-0}. Therefore, the assumption \eqref{eq:vj-f} must fail, which implies \eqref{eq:har-crit} with some constant $C>0$, depending only on $n$ and $m$. This finishes the proof. 
\end{proof}


\subsection{Universal Upper Bounds for $1<\alpha\leq\frac{n+2}{n-2}$}\label{subsection:sup}

With Proposition \ref{proposition:sup-sub}, we obtain a universal upper estimate for (local) singular solutions for the subcritical case. Let us remark that this bound is not sharp for $1<\alpha\leq\frac{n}{n-2}$, although we obtain a universal constant as well as a universal neighborhood in the estimate. The sharp bounds for those cases will be given separately in Section \ref{subsection:asym-} and Section \ref{subsection:asym0}.

\begin{lemma}\label{lemma:u-sup} Let $\u$ be a nonnegative solution of \eqref{eq:main} in $B_1\setminus\{0\}$ with $1<\alpha<\frac{n+2}{n-2}$. Then there exists $C>0$, depending only on $n$, $m$ and $\alpha$, such that
\begin{equation}\label{eq:u-sup}
|\u(x)| \leq C|x|^{-\frac{2}{\alpha-1}} \quad\text{in }B_{1/2}\setminus\{0\}.
\end{equation}
\end{lemma}

\begin{proof} Let $x_0\in B_{1/2}\setminus\{0\}$ and set $r = \frac{1}{2}|x_0|$. Since $\bar{B}_r(x_0)\subset B_1\setminus\{0\}$, one can define 
\begin{equation*}
\v(x) = r^{\frac{2}{\alpha-1}} \u(rx + x_0)\quad\text{in }\bar{B}_1.
\end{equation*}
As $\u$ being a nonnegative solution of \eqref{eq:main} in $B_1\setminus\{0\}$, we see that $\v$ is a nonnegative solution of \eqref{eq:main-reg}. Moreover, $\v$ is continuous up to the boundary of $B_1$. Hence, Proposition \ref{proposition:sup-sub} applies to $\v$ and taking $x=0$ in \eqref{eq:sup} we obtain \begin{equation*}
|\v(0)| \leq C, 
\end{equation*}
which in terms of $\u$ can be rephrased as
\begin{equation*}
|\u(x_0)| \leq Cr^{-\frac{2}{\alpha-1}}.
\end{equation*}
Since $x_0\in B_{1/2}\setminus\{0\}$ was arbitrary and $r=\frac{1}{2}|x_0|$, the proof is finished.
\end{proof}

\begin{remark}\label{remark:u-sup}
For $1<\alpha<\frac{n+2}{n-2}$, one may take an alternative approach as follows. Let $w=u_1+u_2+\dots+u_m$. Then $w\geq 0$ and $\frac{1}{c_1} w\leq |\u|\leq c_1w$ in $B_1\setminus\{0\}$ with $c_1 = m^{\frac{1}{2}}$. Hence, $w$ satisfies $\frac{1}{c_2} w^{\alpha}\leq -\Delta w\leq c_2 w^\alpha$ in $B_1\setminus\{0\}$ with $c_2 = m^{\frac{\alpha-1}{2}}$. By \cite[Corollary IV]{SZ2002} it follows that $w\leq C|x|^{-\frac{2}{\alpha-1}}$ in $B_{1/2}\setminus\{0\}$, where $C$ depends only on $n$, $m$ and $\alpha$. This together with the inequality $|\u| \leq c_1w$ yields \eqref{eq:u-sup}.
\end{remark}

From the Harnack-type inequality in Proposition \ref{proposition:har-crit}, we obtain an upper estimate for the  critical case $\alpha = \frac{n+2}{n-2}$.

\begin{lemma}\label{lemma:u-sup-crit} Let $\u$ be a nonnegative solution of \eqref{eq:main} in $B_1\setminus\{0\}$ with $\alpha = \frac{n+2}{n-2}$. Then there exists $C>0$, depending only on $n$ and $m$, such that
\begin{equation}\label{eq:u-sup-crit}
\left( \min_{i\in I_m} \inf_{\partial B_{3/4}} u_i\right) |\u(x)| \leq C |x|^{-\frac{n-2}{2}}  \quad\text{in }B_{1/2}\setminus\{0\},
\end{equation}
where $I_m$ consists of all indices $1\leq i\leq m$ such that $u_i$ is nontrivial.
\end{lemma}

\begin{proof} If $\u$ has a removable singularity at the origin, then $-\Delta \u = |\u|^{\frac{4}{n-2}} \u$ in $B_1$ (instead of $B_1\setminus\{0\}$), whence one may apply Proposition \ref{proposition:har-crit} to $\u$ after scaling, and observe that 
\begin{equation*}
\left( \min_{i\in I_m} \inf_{\partial B_{3/4}} u_i\right) |\u(x)| \leq C\left(\frac{3}{4}-|x|\right)^{-\frac{n-2}{2}} \leq C\left(\frac{3}{4}\right)^{-\frac{n-2}{2}}  \quad\text{in }B_{1/2}\setminus\{0\},
\end{equation*}
which implies \eqref{eq:u-sup-crit}. 

Henceforth, let us assume that $\u$ does not have a removable singularity at the origin. Clearly $I_m\neq \emptyset$, and by the super-harmonicity and the non-negativity of $u_i$ with $i\in I_m$, we have $u_i > 0$ in $B_1\setminus\{0\}$ for all $i\in I_m$. 

Now let $x_0\in B_{1/2}\setminus\{0\}$ and $r = \frac{1}{8}|x_0|$. Since $\bar{B}_{2r}(x_0)\subset B_1\setminus\{0\}$, one can define
\begin{equation*}
\v(x) = r^{\frac{n-2}{2}} \u(rx + x_0)\quad\text{in }\bar{B}_2. 
\end{equation*}
Obviously, $v_i$ is nontrivial if and only if $i\in I_m$. On the other hand, as $\u$ being a nonnegative solution of \eqref{eq:main} in $B_1\setminus\{0\}$, $\v$ becomes a nonnegative solution of \eqref{eq:main-crit}. Hence, it follows from \eqref{eq:har-crit} that
\begin{equation}\label{eq:v0-sup-crit}
|\v(0)| \leq C\left( \min_{i\in I_m} \inf_{\partial B_2} v_i\right)^{-1} = C \left( \min_{i \in I_m} \inf_{B_{2r}(x_0)} u_i \right)^{-1},
\end{equation}
where $C>0$ depends only on $n$ and $m$. 

Now let $J_m\subset I_m$ consists of all components $u_i$ having non-removable singularity at the origin. Note that $J_m$ may not be equal to $I_m$. By the super-harmonicity and the positivity, the maximum principle implies that $\liminf_{x\ra 0} u_i(x) = \infty$ for each $i\in J_m$. On the other hand, if $i\in I_m\setminus J_m$ (provided that $I_m\setminus J_m \neq \emptyset$), $u_i$ is bounded at the origin, and again by the maximum principle, one has $\liminf_{x\ra 0}u_i(x) \geq \inf_{\partial B_{3/4}} u_i$. Hence, one should have $\inf_{\partial B_{2r}(x_0)} u_i \geq \inf_{\partial B_{3/4}} u_i$ for any $i\in I_m$. This along with \eqref{eq:v0-sup-crit} yields 
\begin{equation*}
|\u(x_0)| \leq C \left( \min_{i\in I_m} \inf_{B_{3/4}} u_i \right)^{-1} r^{-\frac{n-2}{2}},
\end{equation*}
which proves the lemma. 
\end{proof}

\begin{remark}\label{remark:u-sup-crit} We shall obtain later in Section \ref{subsection:asym+crit} without the term in the parenthesis, provided that $\u$ has a non-removable singularity at the origin. 
\end{remark}

Due to Lemma \ref{lemma:u-sup} and Lemma \ref{lemma:u-sup-crit}, we obtain the standard Harnack inequality and interior gradient estimate. 

\begin{lemma}\label{lemma:u-har-Du} Let $\u$ be a nonnegative solution of \eqref{eq:main} in $B_1\setminus\{0\}$ with $1<\alpha\leq\frac{n+2}{n-2}$. Then there exists $C>0$ such that for each $1\leq i\leq m$,
\begin{equation}\label{eq:u-har}
\sup_{B_r\setminus \bar{B}_{r/2}} u_i \leq C \inf_{B_r\setminus \bar{B}_{r/2}} u_i \quad\text{for any }0<r<\frac{1}{2},
\end{equation}
and
\begin{equation}\label{eq:Du-Linf}
|\nabla u_i(x)| \leq C\frac{u_i(x)}{|x|} \quad\text{in }B_{1/2}\setminus\{0\}. 
\end{equation} 
Moreover, the constant $C$ in \eqref{eq:u-har} depends only on $n$, $m$ and $\alpha$, provided that $1<\alpha<\frac{n+2}{n-2}$.
\end{lemma} 

\begin{proof} After a scaling argument we may also say that \eqref{eq:u-sup} and \eqref{eq:u-sup-crit} hold in $B_{3/4}\setminus\{0\}$, instead of $B_{1/2}\setminus\{0\}$. Consider $u_i$, $1\leq i\leq m$, as a nonnegative solution of $-\Delta u_i = a(x)u_i$ in $B_1\setminus\{0\}$, where $a(x)=|\u|^{\alpha-1}$. Due to \eqref{eq:u-sup} if $1<\alpha<\frac{n+2}{n-2}$, and to \eqref{eq:u-sup-crit} if $\alpha=\frac{n+2}{n-2}$, we know that $0\leq a(x)\leq C|x|^{-2}$ in $B_{3/4}\setminus\{0\}$. Thus, \eqref{eq:u-har} follow easily from the classical Harnack inequality \cite[Corollary 9.25]{GT}. With \eqref{eq:u-har} at hand, one may also prove \eqref{eq:Du-Linf} by the classical gradient estimate \cite[Theorem 3.9]{GT}. 
\end{proof}

\section{Asymptotic Radial Symmetry of Local Solutions}\label{section:asym-rad}

This section is devoted to the proof of Theorem \ref{theorem:asym-rad}. Let us address that a similar argument was also used in \cite[Theorem 1.2]{CJSX}, which is concerned with fractional Laplacian, scalar equations. 

\begin{proof}[Proof of Theorem \ref{theorem:asym-rad}]
If the origin is a removable singularity, then the conclusion \eqref{eq:asym-rad} is clear. Hence, we shall assume that the origin is a non-removable singularity. 

Recall from \eqref{eq:kelvin} that $\u_{z,r}^*$ is the Kelvin transform of $\u$ with respect to the sphere $\partial B_r(z)$. Since the origin is a non-removable singularity of $\u$, one may prove, with a minor modification of the proof of Lemma \ref{lemma:rad}, that there is some small $\e>0$ such that for any $z\in B_{\e/2}\setminus\{0\}$ and any $0<r\leq |z|$, 
\begin{equation}\label{eq:asym-rad1}
(u_i)_{z,r}^* \leq u_i\quad\text{in $B_1\setminus (B_r(z)\cup\{0\})$ for each $1\leq i\leq m$}. 
\end{equation}

The key observation here is that \eqref{eq:asym-rad1} implies, for any $a>\frac{1}{\e}$ and $e\in\partial B_1$, 
\begin{equation}\label{eq:asym-rad2}
u_i^*(y) \leq u_i^*(y_a)\quad\text{if $y\cdot e > a$ and $|y_a|>1$ for each $1\leq i\leq m$},
\end{equation}
where 
\begin{equation*}
u_i^*(y) = (u_i)_{0,1}^*(y) = |y|^{2-n} u_i(|y|^{-2}y),\quad y_a = y + 2(a-y\cdot e)e,
\end{equation*}
and $H_a(e)$ is the half-space $\{x:x\cdot e > a\}$. Note that $y_a$ is the reflection point of $y$ with respect to the hyperplane $\partial H_a(e)$. To prove the claim \eqref{eq:asym-rad2}, let us note first that $y\in B_{1/\e}$ if and only if $\frac{y}{|y|^2} \in B_\e$. Now we shall choose some $z\in B_{\e/2}\setminus\{0\}$ and some $0<r<|z|$ such that 
\begin{equation}\label{eq:asym-rad3}
\frac{y_a}{|y_a|^2} - z = \left(\dfrac{r}{\left| \frac{y}{|y|^2} - z\right|}\right)^2 \left( \frac{y}{|y|^2} - z \right).
\end{equation}
In other words, $\frac{y_a}{|y_a|^2}$ is the reflection point of $\frac{y}{|y|^2}$ with respect to $\partial B_r(z)$. We shall ask in addition that
\begin{equation}\label{eq:asym-rad4}
\frac{|y_a|}{|y|}  \leq \frac{1}{r}\left|\frac{y}{|y|^2} - z\right|. 
\end{equation}
Before we actually find such $z$ and $r$, let us verify that along with \eqref{eq:asym-rad3} and \eqref{eq:asym-rad4}, \eqref{eq:asym-rad1} implies \eqref{eq:asym-rad2} as follows. 

Given $y\in\R^n$ such that $y\cdot e> a$ and $|y_a|>1$, and $0<r<|z|<\frac{\e}{2}$ such that \eqref{eq:asym-rad3} and \eqref{eq:asym-rad4} hold, let us write by $x$ and $x_{z,r}^*$ the points $\frac{y}{|y|^2}$ and respectively $\frac{y_a}{|y_a|^2}$. Then since $y\cdot e>a>\frac{1}{\e}$ and $|y_a|>1$, we have $x\in B_r(z)$, and $x_{z,r}^*\in B_1\setminus B_r(z)$. Hence, one may proceed, using \eqref{eq:asym-rad1}, as 
\begin{equation*}
\begin{split}
u_i^*(y) & = \frac{1}{|y|^{n-2}} \left(\frac{|x_{z,r}^* - z|}{r}\right)^{n-2} (u_i)_{z,r}^*(x_{z,r}^*) \\
& \leq \frac{1}{|y|^{n-2}} \left(\frac{|x_{z,r}^* - z|}{r}\right)^{n-2} u_i(x_{z,r}^*) \\
& \leq u_i^*(y_a),
\end{split}
\end{equation*}
proving \eqref{eq:asym-rad2}, where in deriving the first equality we used \eqref{eq:asym-rad3} while the last inequality follows from \eqref{eq:asym-rad4}. Thus, we only need to prove that there actually exist $0<r<|z|<\frac{\e}{2}$ satisfying \eqref{eq:asym-rad3} and \eqref{eq:asym-rad4}. However, it only involves an elementary argument to verify \eqref{eq:asym-rad3} and \eqref{eq:asym-rad4} as well as $0<r\leq |z|<\frac{\e}{2}$, by choosing $r=|z|$ and 
\begin{equation*}
z = \frac{1}{|y|^2}y + \frac{|y_a|^2}{|y|^2 - |y_a|^2}\left( \frac{1}{|y|^2}y- \frac{1}{|y_a|^2}y_a \right) = \frac{1}{|y|^2 - |y_a|^2}(y-y_a).
\end{equation*}

With the claim \eqref{eq:asym-rad2} at hand, one may invoke \cite[Theorem 6.1 and Corollary 6.2]{CGS} to finish the proof. That is, from the former one obtains some $C>0$, independent of $\e$, such that
\begin{equation*}
u_i^*(y) \leq u_i^*(x) \quad\text{if }|x|>1\text{ and }|y|\geq |x| + \frac{C}{\e}\text{ for each $1\leq i\leq m$}.
\end{equation*}
As $u_i^*$ being a nonnegative superharmonic function, the latter implies 
\begin{equation*}
u_i^* = \left( 1+ O\left( \frac{1}{R}\right) \right) \left(\inf_{\partial B_R} u_i^*\right)\quad\text{uniformly on $\partial B_R$ as $R\ra\infty$}, 
\end{equation*}
which in terms of $u_i$ implies the asymptotic radial symmetry claimed as in \eqref{eq:asym-rad}. Hence, the proof is finished. 
\end{proof}

With the asymptotic radial symmetry as well as the uniform estimate achieved in the previous section, we are ready to prove Proposition \ref{proposition:Phi*-const}, finally showing the existence of the second Pohozaev invariant (see \eqref{eq:kappa*}). 

\begin{proof}[Proof of Proposition \ref{proposition:Phi*-const}] Let $\u$ be a nonnegative solution of \eqref{eq:main} in $B_R\setminus\{0\}$ with $\alpha = \frac{n+2}{n-2}$, and let $\Phi_*(r,\u)$ be as in \eqref{eq:Phi*}. To avoid the triviality, let us also assume that $\u$ is a nontrivial solution. Let us prove the well-definedness of $\Phi_*(r,\u)$. 

In the following, we shall denote by $C$ a positive generic constant independent of $r$. With $f(r,\u)$ given as in \eqref{eq:f}, it follows immediately from \eqref{eq:u-sup-crit} and \eqref{eq:Du-Linf} that
\begin{equation}\label{eq:f-Lip}
f(r,\u) \leq C\quad\text{and}\quad r|\dot{f}(r,\u)| \leq C f(r,\u) \quad\text{for any }0<r<\frac{R}{2}.
\end{equation}

On the other hand, by the asymptotic radial symmetry \eqref{eq:asym-rad}, we have
\begin{equation*}
|\Delta (\u - \bar\u)| \leq C|x||\bar\u|^{\frac{n+2}{n-2}}\quad\text{ in } B_{2r}\setminus \bar{B}_r, \quad \text{as } r\ra 0+,
\end{equation*}
where $\bar\u(r)$ is the average of $\u$ over the sphere $\partial B_r$. Hence, it follows from the interior gradient estimate \cite[Theorem 3.9]{GT} and the Harnack inequality \eqref{eq:u-har} that
\begin{equation*}
|\nabla (\u-\bar\u)| \leq C|\u|\quad\text{on } \partial B_r,
\end{equation*}
and in particular,
\begin{equation}\label{eq:Dut-Linf}
|\nabla_\sigma \u | \leq C|\u|\quad\text{on }\partial B_r,
\end{equation}
where $\nabla_\sigma \u$ is the tangential derivative of $\u$ on $\partial B_r$. 

By means of \eqref{eq:Dut-Linf} and \eqref{eq:f-Lip}, we deduce that
\begin{equation}\label{eq:Phi*-1}
\left|\int_0^r \left(\frac{\rho}{n\omega_n}\int_{\partial B_\rho} |\nabla_\sigma \u|^2\,d\sigma\right) \dot{f}(\rho,\u)\,d\rho\right| \leq C\int_0^r \rho f(\rho,\u)^2 \,d\rho,
\end{equation}
provided that $r>0$ is sufficiently small. Similarly, one may also prove from \eqref{eq:asym-rad} and \eqref{eq:f-Lip} that
\begin{equation}\label{eq:Phi*-2}
\left|\int_0^r \left(\frac{\rho}{n\omega_n}\int_{\partial B_\rho} |\u|^{\frac{2n}{n-2}}\,d\rho - f(\rho,\u)^{\frac{n}{n-2}}\right)\dot{f}(\rho,\u)\,d\rho\right| \leq C\int_0^r \rho f(\rho,\u)^{\frac{2n-2}{n-2}} \,d\rho. 
\end{equation}
By the first inequality in \eqref{eq:f-Lip}, we see that the right sides of both \eqref{eq:Phi*-1} and \eqref{eq:Phi*-2} are of order $r^2$, proving the well-definedness of $\Phi_*(r,\u)$. 

Proving that $\Phi_*(r,\u)$ is indeed constant in $0<r<R$ is now easy by considering the cylindrical version $\Psi_*(t,\v)$ defined as in \eqref{eq:Psi*}. Since the computation is very similar with \eqref{eq:Psi*'}, we omit the details. 
\end{proof}


\section{Exact Asymptotic Behavior of Local Solutions}\label{section:asym}

With the {\it a priori} estimates and the classification of the solutions on the punctured space, we are now ready to investigate exact asymptotic behavior of local solutions near the isolated singularity at the origin. Before we begin our analysis, let us provide the basic integrability of the solution. 

\begin{lemma}\label{lemma:dist} Let $\u$ be a nonnegative solution of \eqref{eq:main} in $B_1\setminus\{0\}$ with $\alpha>1$. One has $\u \in L^\alpha(B_1;\R^m)$. In particular, if $\alpha \geq \frac{n}{n-2}$, then $\u$ is a distribution solution of \eqref{eq:main} in $B_1\setminus\{0\}$ in $B_1$, that is,
\begin{equation*}
- \int_{B_1} \u\cdot \Delta \v \,dx= \int_{B_1} |\u|^{\alpha-1}\u\cdot\v \,dx \quad\text{for any }\v \in C_0^\infty(B_1;\R^m). 
\end{equation*}
\end{lemma}

\begin{proof}
Recall from the proof of Proposition \ref{proposition:sup-sub} and Remark \ref{remark:u-sup} that $w = u_1 + \cdots + u_m$ satisfies $\frac{1}{c} w^\alpha \leq -\Delta w \leq c w^\alpha$, with some $c>1$ depending only on $m$ and $\alpha$. By \cite{BL}, $w\in L^\alpha(B_1)$ which implies that $\u\in L^\alpha(B_1;\R^m)$. The second assertion can be proved similarly as in \cite{CGS}, and we omit the details.
\end{proof}

\subsection{Case $\frac{n}{n-2}<\alpha<\frac{n+2}{n-2}$}\label{subsection:asym+sub}

The upper bound \eqref{eq:u-sup} and the classification of solutions on the punctured space allow us to  capture the exact asymptotic behavior of local solutions to \eqref{eq:main}, by means of the blowup analysis. Let us recall from Section \ref{section:sing} that a blowup $\u_0$ is a limit of $\u_r$ along a sequence $r=r_j\ra 0+$ in $C_{loc}^2(\R^n\setminus\{0\};\R^m)$. 

\begin{lemma}\label{lemma:hom} Let $\u$ be a nonnegative solution of \eqref{eq:main} in $B_1\setminus\{0\}$ with $\frac{n}{n-2}<\alpha<\frac{n+2}{n-2}$, and let $\Phi(r,\u)$ be as in \eqref{eq:Phi}. Then $\Phi(0+,\u)\in \{-\bar\lambda,0\}$, where $\bar\lambda$ is given by \eqref{eq:lamb}. Moreover, the following are true.
\begin{enumerate}[(i)]
\item $\Phi(0+,\u) = 0$ if and only if  
\begin{equation}\label{eq:hom-reg}
|\u(x)| = o(|x|^{-\frac{2}{\alpha-1}})\quad\text{as }x\ra 0.
\end{equation}
\item $\Phi(0+,\u) = -\bar\lambda$ if and only if 
\begin{equation}\label{eq:hom-sing}
|\u(x)| = (1+ o(1)) \lambda^{\frac{1}{\alpha-1}} |x|^{-\frac{2}{\alpha-1}}\quad\text{as }x\ra 0,
\end{equation}
where $\lambda$ is given by \eqref{eq:lam-mu}.
\end{enumerate}
\end{lemma}

\begin{proof} Due to the estimates \eqref{eq:u-sup} and \eqref{eq:Du-Linf}, we know that $\Phi(r,\u)$ in \eqref{eq:Phi} is uniformly bounded for all $0<r<\frac{1}{2}$. This combined with the monotonicity (Proposition \ref{proposition:Phi-monot} (i)) implies that $\Phi(0+,\u)$ exists. Hence, we may argue analogously as the proof of Lemma \ref{lemma:u-Phi} and observe that any blowup $\u_0$ of $\u$ satisfies $\Phi(r,\u_0) = \Phi(0+,\u)$ for all $r>0$. As $\u_0$ being a nonnegative solution of \eqref{eq:main} in $\R^n\setminus\{0\}$, it follows from Lemma \ref{lemma:class-sub} (ii) that $\Phi(0+,\u) = 0$ if and only if any blowup $\u_0$ of $\u$ is trivial, while $\Phi(0+,\u) = -\bar\lambda$ if and only if any blowup of $\u_0$ is of the form $\lambda^{\frac{1}{\alpha-1}} |x|^{-\frac{2}{\alpha-1}}\ee$ with some nonnegative unit vector $\ee\in\R^m$. In other words, $\Phi(0+,\u) = 0$ if and only if $|\u_r| \ra 0$ uniformly on $\partial B_1$, while $\Phi(0+,\u) = - \bar\lambda$ if and only if $|\u_r| \ra \lambda^{\frac{1}{\alpha-1}}$ uniformly on $\partial B_1$, where $\u_r$ is the scaling function defined by \eqref{eq:ur}. This finishes the proof. 
\end{proof}

The next lemma shows that \eqref{eq:hom-reg} is sufficient for the origin to be a removable singularity. 

\begin{lemma}\label{lemma:remv+sub} Let $\u$ be a nonnegative solution of \eqref{eq:main} in $B_1\setminus\{0\}$ with $\frac{n}{n-2}<\alpha<\frac{n+2}{n-2}$. If $\u$ satisfies 
\begin{equation}\label{eq:remv+sub}
|\u(x)| = o(|x|^{-\frac{2}{\alpha-1}} )\quad\text{as }x\ra 0,
\end{equation}
then the origin is a removable singularity. 
\end{lemma}

\begin{proof} Under the assumption \eqref{eq:remv+sub}, we claim that 
\begin{equation}\label{eq:remv-re}
|\u(x)| \leq c|x|^{-\frac{2}{\alpha-1} + \delta}\quad\text{in }B_{r_0}\setminus\{0\},
\end{equation}
for some $\delta>0$, $r_0>0$ and  $c>1$, where $c$ and $r_0$ may depend on $\u$. 

Consider the auxiliary function
\begin{equation}\label{eq:vpe}
\vp_\e (x) = ( C_0 r_0^{-\delta} |x|^\delta + \e )|x|^{-\frac{2}{\alpha-1}}\quad\text{in }\R^n\setminus\{0\},
\end{equation}
where $C_0>0$ is the (universal) constant from \eqref{eq:u-sup}, $r_0>0$ is a small radius to be determined later and $\e>0$ is an arbitrary small number. By direct computation, we observe that 
\begin{equation*}
\Delta \vp_\e = - ( C_0r_0^{-\delta} (\lambda + \mu\delta -\delta^2) |x|^\delta + \e \lambda )|x|^{\frac{2\alpha}{1-\alpha}}\quad\text{in }\R^n\setminus\{0\},
\end{equation*}
with $\lambda$ and $\mu$ given by \eqref{eq:lam-mu}. Note that for $\alpha > \frac{n}{n-2} $, we have $\lambda > 0$. Thus, taking $\delta>0$ sufficiently small depending only on $\lambda$ and $|\mu|$, we obtain
\begin{equation}\label{eq:vpe-pde0}
\Delta \vp_\e \leq  -\frac{\lambda}{2|x|^2} \vp_\e\quad\text{in }\R^n\setminus\{0\}.
\end{equation}

Let us fix $1\leq i\leq m$ and consider the $i$-th component $u_i$ of $\u$ as a solution of $\Delta u_i = -a(x) u_i$ in $B_1\setminus\{0\}$ with $a(x) = |\u|^{\alpha-1}$. Due to \eqref{eq:remv+sub}, there exists $r_0>0$ such that $0\leq a(x) \leq \frac{\lambda}{2|x|^2}$ in $B_{r_0}\setminus\{0\}$, and hence, it follows from \eqref{eq:vpe-pde0} that $\vp_\e$ is a supersolution of $\Delta u_i = -a(x) u_i$ in $B_{r_0}\setminus\{0\}$. That is,
\begin{equation}\label{eq:vpe-pde-re}
\Delta \vp_\e \leq - a(x) \vp_\e\quad\text{in }B_{r_0}\setminus\{0\}.
\end{equation}

On the other hand, choosing $C_0>0$ to be the constant for which $|\u|$ satisfies \eqref{eq:u-sup}, we have $u_i \leq C_0r_0^{-\frac{2}{\alpha-1}} \leq \vp_\e$ on $\partial B_{r_0}$. Utilizing the assumption \eqref{eq:remv+sub} again, one can find a sufficiently small $0<r<r_0$ such that $u_i \leq \e |x|^{-\frac{2}{\alpha-1}} \leq \vp_\e$ in $B_r\setminus\{0\}$. Therefore,
\begin{equation}\label{eq:vpe-u-bdry}
u_i \leq \vp_\e\quad\text{on }(\partial B_{r_0}) \cup (B_r\setminus\{0\}).
\end{equation}

In view of \eqref{eq:vpe-pde-re} and \eqref{eq:vpe-u-bdry}, we may apply the maximum principle in $B_{r_0}\setminus B_r$ and obtain $u_i \leq \vp$ in $B_{r_0}\setminus \bar{B}_r$. Combining this inequality with \eqref{eq:vpe-u-bdry}, we arrive at
\begin{equation}\label{eq:vpe-u-int}
u_i \leq \vp_\e\quad\text{in }B_{r_0}\setminus\{0\}.
\end{equation}
Since the parameters $C_0$, $r_0$ and $\delta$ in the definition \eqref{eq:vpe} of $\vp_\e$ are independent of $\e$, we can take $\e\ra 0$ in \eqref{eq:vpe-u-int} and obtain 
\begin{equation*}
u_i(x) \leq C_0r_0^{-\delta}|x|^{-\frac{2}{\alpha-1} + \delta}\quad\text{in }B_{r_0}\setminus\{0\}.
\end{equation*}
Now that this inequality holds for any $1\leq i\leq m$, we arrive at \eqref{eq:remv-re} with $c = C_0r_0^{-\delta}\sqrt{m}$. 

Since $a(x) = |\u|^{\alpha-1}$, we have from \eqref{eq:remv-re} that $0\leq a(x) \leq c|x|^{-2+(\alpha-1)\delta}$ on $B_{r_0}\setminus\{0\}$, which certainly implies $a\in L^{\frac{n}{2-\eta}}(B_1)$ for some small $\eta>0$. According to Lemma \ref{lemma:dist}, $u_i$ satisfies $-\Delta u_i = a(x) u_i$ in $B_1$ in the distributional sense for each $1\leq i\leq m$, whence the classical result by Serrin \cite[Theorem 1]{S} yields that $u_i$ has a removable singularity at the origin. This proves the lemma.
\end{proof}

\begin{remark}\label{remark:remv+sub} One may have noticed that the proof of Lemma \ref{lemma:remv+sub} works for the upper critical case, $\alpha = \frac{n+2}{n-2}$, without any modification. 
\end{remark}

We are ready to prove Theorem \ref{theorem:main} (i).

\begin{proof}[Proof of Theorem \ref{theorem:main}] Suppose that $\u$ has a non-removable singularity at the origin. Then by Lemma \ref{lemma:remv+sub}, $\u$ does not satisfy \eqref{eq:remv+sub}, whence it follows from Lemma \ref{lemma:hom} that $\u$ satisfies \eqref{eq:hom-sing}, which proves \eqref{eq:asym+sub} 
\end{proof}


\subsection{Case $\alpha = \frac{n+2}{n-2}$}\label{subsection:asym+crit}

The asymptotic behavior for the case $\alpha = \frac{n+2}{n-2}$ becomes more subtle, due to the presence of the second Pohozaev invariant $\kappa_*$ given by \eqref{eq:kappa*}. The following lemma is the local version of Theorem \ref{theorem:main-s} (iii). Let us remark that the proof is similar to the classical argument (c.f. the proof of \cite[Theorem 1.2]{CGS}); however, the key difference is that we apply the radial symmetry to the second Pohozaev identity \eqref{eq:kappa*}, instead of the first identity \eqref{eq:kappa}. 

\begin{lemma}\label{lemma:hom+} Let $\u$ be a nonnegative solution of \eqref{eq:main} in $B_1\setminus\{0\}$ with $\alpha=\frac{n+2}{n-2}$. Also set $\kappa(\u)$ and $\kappa_*(\u)$ as in \eqref{eq:kappa} and respectively \eqref{eq:kappa*}. Then $\kappa(\u)$ and $\kappa_*(\u)$ satisfy \eqref{eq:class-crit-kappa} and respectively \eqref{eq:class-crit-kappa*}. Moreover, the following are true. 
\begin{enumerate}[(i)]
\item $\kappa(\u) = \kappa_*(\u) = 0$  if and only if 
\begin{equation}\label{eq:hom+-reg}
|\u(x)| = o(|x|^{-\frac{n-2}{2}}) \quad\text{as }x\ra 0.
\end{equation}
\item $\kappa(\u)^2 + \kappa_*(\u)^2 >0$ if and only if there are $c,C>0$ such that 
\begin{equation}\label{eq:hom+-sing}
c|x|^{-\frac{n-2}{2}} \leq |\u(x)| \leq C|x|^{-\frac{n-2}{2}}\quad\text{as }x\ra 0,
\end{equation}
where $c$ depends on $\u$ while $C$ is determined by $n$ and $m$ only.
\item $\kappa(\u) = -\frac{2}{n} (\frac{n-2}{2})^n$ and $\kappa_*(\u) = 0$ if and only if 
\begin{equation}\label{eq:hom+-hom}
|\u(x)| = (1+ o(1)) \left(\frac{n-2}{2}\right)^{\frac{n-2}{2}} |x|^{-\frac{n-2}{2}}\quad\text{as }x\ra 0.
\end{equation}
\end{enumerate}
\end{lemma}
 
\begin{proof} The existence of $\kappa(\u)$ and $\kappa_*(\u)$ are proved in Proposition \ref{proposition:Phi-monot} (ii) and respectively Proposition \ref{proposition:Phi*-const}. Now let $\u_0$ be any blowup of $\u$, and write $r_j\ra 0+$ by the blowup sequence. By the scaling relation \eqref{eq:Phi-scale} of $\Phi(r,\u)$, we see that
\begin{equation*}
\kappa(\u_0) = \Phi(1,\u_0) = \lim_{j\ra\infty} \Phi(1,\u_{r_j}) = \lim_{j\ra\infty} \Phi(r_j,\u) = \kappa(\u). 
\end{equation*}
However, $\u_0$ is a nonnegative solution of \eqref{eq:main} (with $\alpha = \frac{n+2}{n-2}$) in $\R^n\setminus\{0\}$, whence Lemma \ref{lemma:class-crit} yields $\kappa(\u_0)$ satisfies \eqref{eq:class-crit-kappa}, and so does $\kappa(\u)$. Similarly, one may deduce from the scaling relation \eqref{eq:Phi*-scale} of $\Phi_*(r,\u)$ that $\kappa_*(\u) = \kappa_*(\u_0)$, and by Lemma \ref{lemma:class-crit}, $\kappa_*(\u)$ verifies \eqref{eq:class-crit-kappa*}.

Suppose that $\kappa(\u) = \kappa_*(\u) = 0$, and let $\v$ be the cylindrical transformation of $\u$ as in \eqref{eq:cyl}. Rephrasing the estimates \eqref{eq:Phi*-1} and \eqref{eq:Phi*-2} in terms of $\v$, the second Pohozaev identity \eqref{eq:kappa*-cyl} becomes (as $t\ra\infty$), 
\begin{equation}\label{eq:kappa*-cyl2}
(g')^2 = (n-2)^2 g^2 - \frac{4(n-2)}{n} g^{\frac{2n-2}{n-2}} +  O \left( \int_t^\infty e^{-2\tau} g(\tau)^2\,d\tau\right),
\end{equation}
where $g$ is given by \eqref{eq:g} and $g'=dg/dt$. Since the term $O(\int_t^\infty e^{-2\tau} g(\tau)^2\,d\tau)$ decays exponentially, and is comparably smaller than $g(t)$, the behavior of $g'$ is determined by the nonnegative roots of 
\begin{equation*}
(n-2)^2g^2 - \frac{4(n-2)}{n} g^{\frac{2n-2}{n-2}} = 0,
\end{equation*}
which are $0$ and $(\frac{n(n-2)}{4})^{\frac{n-2}{2}}$ respectively. In particular, $g(t)$ must be either non-increasing and converging to $0$, or nondecreasing and converging to $(\frac{n(n-2)}{4})^{\frac{n-2}{2}}$. 

If $g(t)\ra 0$ as $t\ra\infty$, then by the asymptotic radial symmetry we have $|\v(t,\cdot)| \ra 0$ uniformly on $\Ss^{n-1}$ as $t\ra \infty$. After the inverse cylindrical transform via \eqref{eq:cyl}, we arrive at \eqref{eq:hom+-reg}, as desired. 

Now let us show that the other alternative, i.e., $g(t)\ra (\frac{n(n-2)}{4})^{\frac{n-2}{2}}$ as $t\ra\infty$,  cannot occur. Suppose that this is true. Then again from the asymptotic radial symmetry it follows that $|\u_r|\ra (\frac{n(n-2)}{4})^{\frac{n-2}{2}}$ uniformly on $\partial B_1$ as $r\ra 0+$. This implies that any blowup $\u_0$ of $\u$ must be of the form $(\frac{n(n-2)}{4})^{\frac{n-2}{2}}|x|^{-\frac{n-2}{2}}\ee$ for some nonnegative unit vector $\ee\in\R^m$. In particular, $\u_0$ has a non-removable singularity at the origin, and hence Theorem \ref{theorem:main-s} (iii) yields that $\kappa(\u_0)$ or $\kappa_*(\u_0)$ is non-zero, a contradiction to $\kappa(\u) = \kappa(\u_0) = 0$ or, respectively, $\kappa_*(\u) = \kappa_*(\u_0) = 0$. Hence, the assertion (i) is proved. 

Now let us consider the case when $\kappa(\u)^2 + \kappa_*(\u)^2>0$. Let $\u_0$ be any blowup of $\u$. Then due to the asymptotic radial symmetry of $\u$, $\u_0$ is radially symmetric on the punctured space. Hence, by Lemma \ref{lemma:u-s}, we know that $|\u_0| \leq C|x|^{-\frac{n-2}{2}}$ where $C>0$ depends only on $n$ and $m$. Since $\u_0$ is an arbitrary blowup of $\u$, this proves the upper bound in \eqref{eq:hom+-sing}. 

On the other hand, by Theorem \ref{theorem:main-s} (iii)-(b), the cylindrical transform $\v_0$ of $\u_0$ satisfies \eqref{eq:|v|-poho}. Due to R. H. Fowler \cite{F}, $|\v_0|$ has to be bounded uniformly away from zero, with the bound determined solely on the value of $n$, $\kappa(v_0) =\kappa(u_0) = \kappa(u)$ and $\kappa_*(\v_0) = \kappa_*(\u)$. This proves that $|\u_0| \geq c|x|^{-\frac{n-2}{2}}$ for some $c>0$ depending only on $n$, $\kappa(\u)$ and $\kappa_*(\u)$. Since $c$ is independent of the blowup $\u_0$, the lower bound in \eqref{eq:hom+-sing} is proved. Thus, the assertion (ii) is proved. 

The final assertion regarding \eqref{eq:hom+-hom} follows immediately from Theorem \ref{theorem:main-s} (iii)-(c), since the latter implies that the blowup of $\u$ is unique and is of the form \eqref{eq:u-s-2}, if and only if $\kappa(\u) = -\frac{2}{n}(\frac{n-2}{2})^n$ and $\kappa_*(\u) = 0$. 
\end{proof}

As with Lemma \ref{lemma:remv+sub}, we observe that \eqref{eq:hom+-reg} is a sufficient condition to have a removable singularity. 

\begin{lemma}\label{lemma:remv+crit} Let $\u$ be a nonnegative solution of \eqref{eq:main} in $B_1\setminus\{0\}$ with $\alpha=\frac{n+2}{n-2}$. If $\u$ satisfies 
\begin{equation*}
|\u(x)| = o(|x|^{-\frac{n-2}{2}} )\quad\text{as }x\ra 0,
\end{equation*}
then the origin is a removable singularity. 
\end{lemma}

\begin{proof} As mentioned in Remark \ref{remark:remv+sub}, the same proof of Lemma \ref{lemma:remv+sub} works here as well, whence we leave out the details to the reader.  
\end{proof}

\begin{proof}[Proof of Theorem \ref{theorem:main} (ii)]
Suppose that the origin is a non-removable singularity, and let us write by $\kappa$ and $\kappa_*$ the first and respectively the second Pohozaev invariant. As a contraposition to Lemma \ref{lemma:remv+crit}, \eqref{eq:remv+sub} fails. Thus, by Lemma \ref{lemma:hom+}, one has $\kappa^2 + \kappa_*^2 >0$. Then the asymptotic bounds in \eqref{eq:asym+crit} follows from the second alternative, \eqref{eq:hom+-sing}, of Lemma \ref{lemma:hom+}, and the proof is finished.
\end{proof}


\subsection{Case $1<\alpha<\frac{n}{n-2}$}\label{subsection:asym-}

The asymptotic analysis for the case $1<\alpha<\frac{n}{n-2}$ is very simple. It is noticeable that the monotonicity formula is not required here. We also mention that one can reduce our study to the scalar case by considering $w = u_1+ u_2 + \cdots + u_m \geq 0$, and directly apply the results in \cite{L}. Nevertheless, we shall give a more direct proof, for the sake of completeness. 

We shall begin with the sharp upper estimate.

\begin{lemma}\label{lemma:sup-} Let $\u$ be a nonnegative solution of \eqref{eq:main} in $B_1\setminus\{0\}$ with $1<\alpha<\frac{n}{n-2}$. Then there is $C>0$, depending only $|\u|$, such that
\begin{equation}\label{eq:sup-}
|\u(x)| \leq C|x|^{2-n}\quad\text{as } x\ra 0. 
\end{equation}
\end{lemma}

\begin{proof} Lemma \ref{lemma:dist} asserts that $\u\in L^\alpha(B_1)$. Since $1<\alpha<\frac{n}{n-2}$ and $\u$ satisfies the Harnack inequality \eqref{eq:u-har}, it is easy to verify that
\begin{equation}\label{eq:u-sup-0}
|\u(x)| = o(|x|^{-\frac{2}{\alpha-1}})\quad\text{as }x\ra 0. 
\end{equation}
Utilizing \eqref{eq:u-sup-0}, and noting that $n-2<\frac{2}{\alpha-1}$, one may argue with a blowup argument to prove that for any $n-2<q<\frac{2}{\alpha-1}$, there is some $0<r_q<1$, depending only on $n$, $m$, $\alpha$ and $q$, such that 
\begin{equation}\label{eq:u-sup-q}
|\u(x)| < |x|^{-q}\quad\text{in }B_{r_q}\setminus\{0\}.
\end{equation}

Now let $r_q$ be as in \eqref{eq:u-sup-q}. Due to Lemma \ref{lemma:dist} again, $\Delta \u = - |\u|^{\alpha-1}\u\in L^1(B_1)$, whence one can decompose $\u$, in $B_{r_q}\setminus\{0\}$, as 
\begin{equation}\label{eq:u-exp}
\u(x) = |x|^{2-n}\a - \int_{B_{r_q}} |x-y|^{2-n} \Delta \u(y)\,dy + \h(x),
\end{equation}
where $\a$ is a nonnegative vector in $\R^m$ and $\h$ is a nonnegative and harmonic, vectorial function on $B_{r_q}$. However, owing to the estimate \eqref{eq:u-sup-q}, it is not hard to see from the equation $\Delta \u = -|\u|^{\alpha-1}\u$ that there is $C_q>0$, depending only on $n$, $m$, $\alpha$ and $q$, such that
\begin{equation}\label{eq:newton}
\left| \int_{B_{r_q}} |x-y|^{2-n} \Delta \u(y)\,dy \right| \leq \int_{B_{r_q}} |x-y|^{2-n} |y|^{-\alpha q}\,dy \leq C_q|x|^{2-n}. 
\end{equation}
Thus, choosing $n-2<q<\frac{2}{\alpha-1}$ so as to depend only on $n$ and $\alpha$, and selecting $r_q$ and $C_q$ in \eqref{eq:newton} correspondingly, we derive the sharp estimate \eqref{eq:sup-} from \eqref{eq:u-exp}. 
\end{proof}

Next we consider a sufficient condition to have a removable singularity. 

\begin{lemma}\label{lemma:remv-} Let $\u$ be a nonnegative solution of \eqref{eq:main} in $B_1\setminus\{0\}$ with $1<\alpha<\frac{n}{n-2}$. If $\u$ satisfies
\begin{equation}\label{eq:remv-}
|\u(x)| = o( |x|^{2-n} ) \quad\text{as }x\ra 0,
\end{equation}
then the origin is a removable singularity.
\end{lemma}

\begin{proof} Under the assumption \eqref{eq:remv-}, one has $\u\in L^q(B_1;\R^m)$ for any $1\leq q<\frac{n}{n-2}$. Since $1<\alpha<\frac{n}{n-2}$ and $|\Delta \u| \leq |\u|^\alpha$, we have $-\Delta \u \in L^{\frac{q}{\alpha}}(B_1;\R^m)$ for any $\alpha<q<\frac{n}{n-2}$. Thus, the $L^p$ theory \cite[Theorem 9.9]{GT} (applied to each component of $\u$) and a bootstrap argument based on the Sobolev inequality yields $\u \in W^{2,p} (B_1;\R^m)$ for any $1<p<\infty$. In particular, it follows from the Sobolev embedding that $\u\in C^{1,\gamma}(B_1;\R^m)$ for any $0<\gamma<1$, and thus $\u$ must have a removable singularity at the origin. 
\end{proof} 

We are in a position to prove Theorem \ref{theorem:main} (iii). 

\begin{proof}[Proof of Theorem \ref{theorem:main} (iii)] Suppose that $\u$ has a non-removable singularity at the origin. By Lemma \ref{lemma:remv-}, we know that $\u$ does not satisfy \eqref{eq:remv-}, or equivalently, there is some $\delta>0$, a component, say $u_1$, and a sequence $r_j\ra 0+$ such that
\begin{equation*}
\sup_{\partial B_{r_j}} u_1\geq \delta r_j^{2-n}. 
\end{equation*}
By the Harnack inequality \eqref{eq:u-har}, we know that
\begin{equation*}
\inf_{\partial B_{r_j}} u_1 \geq c_0\delta r_j^{2-n},
\end{equation*}
where $c_0>0$ depends only on $n$, $m$ and $\alpha$. Taking $\delta>0$ smaller, if necessary, such that $c\delta \leq \inf_{\partial B_{1/2}} u_1$, it follows from the maximum principle that 
\begin{equation*}
u_1(x) \geq c_0\delta |x|^{2-n}\quad\text{in }B_{1/2}\setminus\{0\},
\end{equation*} 
proving the asymptotic lower bound in \eqref{eq:asym-}. The asymptotic upper bound in \eqref{eq:asym-} is established in Lemma \ref{lemma:sup-}. Hence, the theorem is proved. 
\end{proof}

\begin{remark}\label{remark:asym-} As mentioned in the beginning of this section, the proof of Theorem \ref{theorem:main} (iii) can also be deduced by considering the function $w = u_1+u_2 + \cdots + u_m \geq 0$. Then $w$ satisfies $C_1w^\alpha \leq -\Delta w\leq C_2 w^\alpha$ in $B_1\setminus\{0\}$, where $C_1,C_2>0$ depend on $n$, $m$ and $\alpha$ only, and the claim in Theorem \ref{theorem:main} (iii) follows now from existent results in the literature, such as \cite[Theorem 2 and Remark 2]{L}. 
\end{remark}


\subsection{Case $\alpha = \frac{n}{n-2}$}\label{subsection:asym0}

The analysis of the lower critical exponent, $\alpha = \frac{n}{n-2}$, exhibits its own subtlety, due to the multiplicity of components in \eqref{eq:main}, as with the upper critical case, $\alpha = \frac{n+2}{n-2}$. To briefly discuss this point, let us first give the asymptotic upper bound. 

\begin{lemma}[Lemma 1 in \cite{A2}]\label{lemma:ub-sup-0} Let $\u$ be a nonnegative solution of \eqref{eq:main} with $\alpha = \frac{n}{n-2}$ in $B_1\setminus\{0\}$. Then for each $1\leq i\leq m$, 
\begin{equation}\label{eq:ub-sup-0}
\bar{u}_i(r) \leq \left(\frac{(n-2)^2}{2}\right)^{\frac{n-2}{2}} r^{2-n} (-\log r)^{\frac{2-n}{2}}\quad\text{as }r\ra 0,
\end{equation}
where $\bar{u}_i$ is the average of $u_i$ over the sphere $\partial B_r$. 
\end{lemma}

\begin{proof} Note that for each $1\leq i\leq m$, $\bar{u}_i$ satisfies, for $0<r<1$,  
\begin{equation*}
\dot{\bar{u}}_i + \frac{n-1}{r} \dot{\bar{u}}_i + \bar{u}_i^{\frac{n}{n-2}} = 0,
\end{equation*}
whence the conclusion follows directly from \cite[Lemma 1]{A2}. 
\end{proof}

Let us remark that the constant $(\frac{1}{2}(n-2)^2)^{\frac{n-2}{2}}$ in \eqref{eq:ub-sup-0} is exact in view of \eqref{eq:asym0}. Due to the fact that $\u$ consists of multiple components, there is not an easy way to prove that $|\bar\u|$ also satisfies \eqref{eq:ub-sup-0} with exactly the same constant. This prevents us from applying the argument in \cite[Section 2]{A2}, which deals with the scalar version of \eqref{eq:main} with $\alpha = \frac{n}{n-2}$. Instead, we mainly follow \cite[Section 3]{A2}, where a sign-changing problem is considered. The idea is to consider several refinements of the usual monotonicity formula $\Psi(t,\v)$ introduced  in \eqref{eq:Psi}. 

Due to the refined upper bound \eqref{eq:ub-sup-0}, we shall consider a new cylindrical transformation $\vphi$ defined so as to satisfy
\begin{equation}\label{eq:cyl-n}
\u(x) = |x|^{2-n} (-\log |x|)^{\frac{2-n}{2}} \vphi\left(-\log |x|,\frac{x}{|x|}\right). 
\end{equation}
Then the problem \eqref{eq:main} (with $\alpha = \frac{n}{n-2}$) can be reformulated in terms of $\vphi$ as
\begin{equation}\label{eq:main-cyl-n}
\partial_{tt} \vphi + (n-2) \left( 1 - \frac{1}{t} \right) \partial_t \vphi + \Delta_\theta \vphi = \frac{n-2}{2t} \left(n-2 -\frac{n}{2t}\right) \vphi - \frac{1}{t} |\vphi|^{\frac{2}{n-2}} \vphi.
\end{equation}

\begin{remark}\label{remark:cyl-n} Due to the asymptotic radial symmetry \eqref{eq:asym-rad} of $\u$, $\vphi$ satisfies $|\vphi - \bar\vphi| = O(e^{-\gamma t})$ as $t\ra\infty$, for some $\gamma >0$, where $\bar\vphi(t)$ is the average of $\vphi(t,\theta)$ over $\theta\in\Ss^{n-1}$. In particular, one has (by arguing similarly as in the derivation of \eqref{eq:Dut-Linf})
\begin{equation}\label{eq:w-tan-Linf}
|\nabla_\theta \vphi(t,\theta)| \leq C e^{-\gamma t}\quad\text{in }(t_0,\infty)\times\Ss^{n-1},
\end{equation}
for some large $t_0$ and $C$ independent of $t$. Moreover, it follows from the sharp estimate \eqref{eq:ub-sup-0} and the gradient estimate \eqref{eq:Du-Linf} that
\begin{equation}\label{eq:w-dt-Linf}
|\vphi(t,\theta)| + |\partial_t \vphi(t,\theta)| \leq C\quad\text{in }(t_0,\infty)\times\Ss^{n-1}. 
\end{equation}
\end{remark}

In comparison with \eqref{eq:main-cyl}, we obtain the first refinement of the monotonicity formula $\Psi(t,\v)$, given as 
\begin{equation}\label{eq:E}
\begin{split}
E(t,\vphi) &= \frac{1}{n\omega_n} \int_{\Ss^{n-1}} \left( t |\partial_t \vphi|^2 - t|\nabla_\theta \vphi|^2 + \frac{n-2}{n-1} |\vphi|^{\frac{2n-2}{n-2}} \right)\,d\theta \\
&\quad - \frac{n-2}{2n\omega_n} \left( n -2 - \frac{n}{2t} \right)\int_{\Ss^{n-1}} |\vphi|^2\,d\theta.
\end{split}
\end{equation}
Note that $E(t,\vphi)$ is well-defined for any $t$ whenever $\vphi(t,\cdot)$ is defined on $\Ss^{n-1}$, due to the smoothness of $\u$.

The next lemma is concerned with the monotonicity of $E(t,\vphi)$. 

\begin{lemma}\label{lemma:E-monot} Let $\u$ be a nonnegative solution of \eqref{eq:main} in $B_1\setminus\{0\}$ with $\alpha = \frac{n}{n-2}$, and $\vphi$ be the cylindrical transformation as in \eqref{eq:cyl-n}. Then
\begin{equation}\label{eq:E'}
\begin{split}
E'(t,\vphi) &= - \frac{(2n-4)t-2n+3}{n\omega_n} \int_{\Ss^{n-1}} |\partial_t \vphi|^2 \,d\theta \\
&\quad - \frac{1}{n\omega_n}\int_{\Ss^{n-1}} \left( |\nabla_\theta \vphi|^2 \,d\theta + \frac{n(n-2)}{4t^2}  |\vphi|^2 \right)\,d\theta. 
\end{split}
\end{equation}
In particular, $E(t,\vphi)$ is nonincreasing for $t>\frac{2n-3}{2n-4}$, and $E(\infty,\vphi)$ exists. 
\end{lemma}

\begin{proof} The proof of \eqref{eq:E'} follows easily from taking the inner product of \eqref{eq:main-cyl-n} with $t\partial_t\vphi$ and integrating the both sides over $\Ss^{n-1}$. We omit the details. 

With \eqref{eq:E'} at hand, we know that $E(t,\vphi)$ is nonincreasing for $t>\frac{2n-3}{2n-4}$. Thus, the existence of $E(\infty,\vphi)$ follows immediately from that $E(t,\vphi)$ is uniformly bounded from below as $t\ra\infty$. However, \eqref{eq:w-tan-Linf} yields 
\begin{equation*}
\lim_{t\ra\infty}\int_{\Ss^{n-1}} t|\nabla_\theta\vphi|^2\,d\theta = 0,
\end{equation*}
which along with \eqref{eq:w-dt-Linf} ensures that 
\begin{equation*}
\liminf_{t\ra\infty} E(t,\vphi) > -\infty,  
\end{equation*}
as desired. 
\end{proof}

In order to have the full strength of the existence of $E(\infty,\vphi)$, we shall prove the following, which is the system version of \cite[Lemma 3.2]{A2}. Although the proof is almost identical, we shall present the argument for the sake of completeness. 

\begin{lemma}[Essentially due to \cite{A2}]\label{lemma:w-dt-int} Let $\vphi$ be as in Lemma \ref{lemma:E-monot}. Then 
\begin{equation}\label{eq:w-dt-int}
\lim_{t\ra\infty} \int_{\Ss^{n-1}} t|\partial_t \vphi|^2 \,d\theta = 0. 
\end{equation}
\end{lemma}

\begin{proof} By \eqref{eq:w-tan-Linf} and \eqref{eq:w-dt-Linf}, one may integrate the both sides of \eqref{eq:E'} from $t_0 = \frac{2n-3}{2n-4}$ to $\infty$, and use the existence of $E(\infty,\phi)$ to deduce that
\begin{equation}\label{eq:w-dt-int1}
\int_{t_0}^\infty \int_{\Ss^{n-1}} \tau |\partial_\tau \vphi|^2 \,d\theta \,d\tau < \infty. 
\end{equation}
Hence, it is sufficient to prove that $\int_{\Ss^{n-1}} t|\partial_t\vphi|^2\,d\theta$ is a Cauchy sequence in $t\ra \infty$.

In order to do so, we differentiate \eqref{eq:main-cyl-n} in $t$ and find that $\vpsi = \partial_t\vphi$ solves 
\begin{equation}\label{eq:psi-pde}
\begin{split}
&\partial_{tt} \vpsi + (n-2)\left(1-\frac{1}{t}\right) \partial_t\vpsi - \frac{n-2}{2t} \left(n - 2 - \frac{n+4}{2t} \right) \vpsi + \Delta_\theta \vpsi\\
&= -\frac{n-2}{2t^2} \left( n -2 - \frac{n}{t}\right)\vphi + \frac{1}{t} |\vphi|^{\frac{2}{n-2}}\left( \frac{1}{t} \vphi - \frac{2}{n-2} \frac{\vphi\cdot \vpsi}{|\vphi|^2}\vphi  - \vpsi \right).
\end{split}
\end{equation}
Taking the inner product of \eqref{eq:psi-pde} with $t\partial_t\vpsi$ and integrating over $\Ss^{n-1}$, one may verify after some computation that the functional 
\begin{equation}\label{eq:J}
\begin{split}
J(t,\vpsi) &= \frac{1}{n\omega_n}\int_{\Ss^{n-1}} \left(t|\partial_t\vpsi|^2 - t|\nabla_\theta \vpsi|^2 - \frac{n-2}{2} \left( n-2 - \frac{n+4}{2t}\right) |\vpsi|^2\right) \, d\theta \\
&\quad  - \frac{1}{n\omega_n}\int_t^\infty \int_{\Ss^{n-1}} \frac{n-2}{\tau}\left(n-2 - \frac{n}{\tau}\right) \vphi\cdot \partial_\tau\vpsi \, d\theta\,d\tau \\
&\quad  + \frac{1}{n\omega_n} \int_t^\infty \int_{\Ss^{n-1}} |\vphi|^{\frac{2}{n-2}} \left( \frac{1}{\tau}\vphi- \frac{2}{n-2} \frac{\vphi\cdot\vpsi}{|\vphi|^2} \vphi - \vpsi \right)\cdot \partial_\tau\vpsi\,d\theta\,d\tau
\end{split}
\end{equation}
satisfies 
\begin{equation}\label{eq:J'}
\begin{split}
J'(t,\vpsi) &=-\frac{(2n-4)t - 2n+3}{n\omega_n} \int_{\Ss^{n-1}} |\partial_t\vpsi|^2\,d\theta \\
&\quad - \frac{1}{n\omega_n} \int_{\Ss^{n-1}} \left( |\nabla_\theta \vpsi|^2 + \frac{(n+4)(n-2)}{t^2} \int_{\Ss^{n-1}} |\vpsi|^2 \right) \,d\theta,
\end{split}
\end{equation}
provided that the last two double integrals in \eqref{eq:J} are finite, i.e., $J(t,\vpsi)$ is well-defined for all $t$ large. 

Assuming for the moment that $J(t,\psi)$ is well-defined for all $t$ large, one may proceed as in the proof of \cite[Lemma 3.2]{A2}. Note that \eqref{eq:J'} implies the monotonicity of $J(t,\vpsi)$ for $t\geq t_0 = \frac{2n-3}{2n-4}$. Analogously with Remark \ref{remark:cyl-n}, the asymptotic radial symmetry \eqref{eq:asym-rad} implies exponential decay of $|\nabla_\theta \vpsi|$ as well as the uniform boundedness of $|\vpsi|$ and $|\partial_t\vpsi|$. Hence, one may deduce as in the proof of Lemma \ref{lemma:E-monot} that $J(t,\vpsi)$ is uniformly bounded from below as $t\ra\infty$. As $J(t,\vpsi)$ being nonincreasing in $t\geq t_0$, $J(\infty,\vpsi)$ exists, and thus, integrating \eqref{eq:J} from $t_0$ to $\infty$ yields that
\begin{equation}\label{eq:w-dt-int2}
\int_{t_0}^\infty \int_{\Ss^{n-1}} \tau |\partial_\tau \vpsi|^2 \,d\theta \,d\tau < \infty. 
\end{equation}
Noting that 
\begin{equation*}
\left|\frac{d}{dt} \left( t \int_{\Ss^{n-1}} |\partial_t\vphi|^2\,d\theta\right)\right| \leq \int_{\Ss^{n-1}} (|\partial_t\vphi|^2 + t |\partial_t \vphi|^2 + t|\partial_{tt}\vphi|^2) \,d\theta,
\end{equation*}
we conclude from \eqref{eq:w-dt-int1} and \eqref{eq:w-dt-int2} that $t \int_{\Ss^{n-1}} |\partial_t\vphi|^2\,d\theta$ is a Cauchy sequence in $t\ra\infty$. Thus, \eqref{eq:w-dt-int} follows from \eqref{eq:w-dt-int1}. 

To this end, we are only left with verifying the well-definedness of $J(t,\vpsi)$ for all $t\geq t_0$ with some $t_0$ large. As noted above, this boils down to proving that the last two double integrals in \eqref{eq:J} are finite. Due to the upper estimate \eqref{eq:ub-sup-0} and \eqref{eq:w-dt-int1}, it suffices to show that
\begin{equation}\label{eq:w-dt-int3}
\int_{t_0}^\infty \frac{1}{t} \int_{\Ss^{n-1}} (|\vphi| + |\vpsi|) |\partial_t \vpsi| \,d\theta\,dt < \infty.
\end{equation}

Owing to \eqref{eq:w-tan-Linf} and \eqref{eq:w-dt-Linf}, we have, in \eqref{eq:main-cyl-n} (recall that $\vpsi = \partial_t\vphi$), 
\begin{equation}\label{eq:main-cyl-n-re}
|\partial_t\vpsi|=  (n-2)|\vpsi| + O\left(\frac{1}{t}\right), 
\end{equation}
so multiplying \eqref{eq:main-cyl-n-re} by $\frac{1}{t} |\vphi|$ yields 
\begin{equation}\label{eq:w-dt-int4}
\begin{split}
\int_{t_0}^\infty \frac{1}{t} \int_{\Ss^{n-1}} |\vphi||\partial_t \vpsi|\,d\theta \,dt &\leq (n-2)\int_{t_0}^\infty \frac{1}{t} \int_{\Ss^{n-1}} |\vphi| |\vpsi| \,d\theta \,dt + O(1) \\
& \leq \frac{n-2}{2} \int_{t_0}^\infty \int_{\Ss^{n-1}} |\vpsi|^2\,d\theta\,dt + O(1) \\
& < \infty, 
\end{split}
\end{equation}
where the second inequality follows from $|\vphi||\vpsi| \leq \frac{1}{2t}|\vphi|^2 + \frac{t}{2} |\vpsi|^2$, while the last inequality is derived from \eqref{eq:w-dt-int1}.  On the other hand, multiplying \eqref{eq:main-cyl-n-re} by $\frac{1}{t} |\vpsi|$, we deduce from \eqref{eq:w-dt-int1} that
\begin{equation}\label{eq:w-dt-int5}
\begin{split}
\int_{t_0}^\infty \frac{1}{t} \int_{\Ss^{n-1}} |\vpsi| |\partial_t \vpsi|\,d\theta\,dt \leq (n-2) \int_{t_0}^\infty \frac{1}{t} \int_{\Ss^{n-1}} |\vpsi|^2\,d\theta\,dt < \infty. 
\end{split}
\end{equation}
The claim \eqref{eq:w-dt-int3} follows readily from \eqref{eq:w-dt-int4} and \eqref{eq:w-dt-int5}. The  proof is finished. 
\end{proof}

Finally we have the classification of the blowup limit via the limiting energy levels $E(\infty,\vphi)$. 

\begin{lemma}\label{lemma:class0} Let $\u$ be a nonnegative solution of \eqref{eq:main} in $B_1\setminus\{0\}$ with $\alpha=\frac{n}{n-2}$, and $\vphi$ be its cylindrical transform as in \eqref{eq:cyl-n}. Also let $E(t,\vphi)$ be  as in \eqref{eq:E}. Then $E(\infty,\vphi)\in \{-\frac{1}{n-1}(\frac{(n-2)^2}{2})^{n-1},0\}$. Moreover, the following are true.
\begin{enumerate}[(i)]
\item $E(\infty,\vphi) = 0$ if and only if
\begin{equation}\label{eq:class0-reg}
|\u(x)| = o\left(|x|^{2-n} (-\log |x|)^{\frac{2-n}{2}}\right)\quad\text{as }x\ra 0.
\end{equation}
\item $E(\infty,\vphi) = -\frac{1}{n-1}(\frac{(n-2)^2}{2})^{n-1}$ if and only if
\begin{equation}\label{eq:class0-sing}
|\u(x)| = (1+o(1)) \left( \frac{(n-2)^2}{2} \right)^{\frac{n-2}{2}}|x|^{2-n} (-\log |x|)^{\frac{2-n}{2}}. 
\end{equation}
\end{enumerate}
\end{lemma}

\begin{proof} Due to Lemma \ref{lemma:E-monot}, \eqref{eq:w-tan-Linf} and \eqref{eq:w-dt-int}, we have
\begin{equation}\label{eq:E-inf1}
E(\infty,\vphi) = \frac{1}{n\omega_n} \lim_{t\ra\infty} \int_{\Ss^{n-1}} \left( \frac{n-2}{n-1} |\vphi|^{\frac{2n-2}{n-2}} - \frac{(n-2)^2}{2} |\vphi|^2\right) \,d\theta. 
\end{equation}
In fact, \eqref{eq:w-tan-Linf} implies that whenever $\vphi(t_j,\theta)$ converges as $t_j\ra\infty$, the limit is independent of $\theta\in\Ss^{n-1}$. Hence, along a convergent sequence $\vphi(t_j,\theta) \ra \a$ (uniformly over $\theta\in\Ss^{n-1}$), we obtain from \eqref{eq:E-inf1} that
\begin{equation}\label{eq:E-inf2}
E(\infty,\vphi) = \frac{n-2}{n-1} |\a|^{\frac{2n-2}{n-2}} - \frac{(n-2)^2}{2} |\a|^2. 
\end{equation}
Since the right hand side has at most three nonnegative roots, we conclude that the limit value $|\a|$ (under the uniform convergence of $|\vphi(t,\theta)|$ on $\Ss^{n-1}$ as $t\ra\infty$) is unique. 

To compute the limit value $|\a|$, let us take the inner product of \eqref{eq:main-cyl-n} with $\vphi$ and integrate the both sides over $(t_0,\infty)\times\Ss^{n-1}$ (with $t_0$ large). Then one may easily deduce from \eqref{eq:w-tan-Linf}, \eqref{eq:w-dt-Linf} and \eqref{eq:w-dt-int1} that 
\begin{equation*}
\begin{split}
&\left| \int_{t_0}^\infty \frac{1}{n\omega_n \tau} \int_{\Ss^{n-1}} \left(\frac{(n-2)^2}{2}  - |\vphi|^{\frac{2}{n-2}}\right)|\vphi|^2 \,d\theta\,dt \right| <\infty. 
\end{split}
\end{equation*}
Now that $|\vphi|$ converges to $|\a|$ as $t\ra\infty$ uniformly on $\Ss^{n-1}$, we must have either $|\a| = 0$ or $|\a| = (\frac{(n-2)^2}{2})^{\frac{n-2}{2}}$. Inserting this into \eqref{eq:E-inf2}, we deduce that either $E(\infty,\vphi) = 0$ if and only if $|\a| = 0$, or $E(\infty,\vphi) = -\frac{1}{n-1} (\frac{(n-2)^2}{2})^{n-1}$. Obviously, the assertions \eqref{eq:class0-reg} and \eqref{eq:class0-sing} follow immediately via inverse cylindrical transform \eqref{eq:cyl-n}.
\end{proof}

We are only left with proving that \eqref{eq:class0-reg} yields the removability of the singularity at the origin. 

\begin{lemma}\label{lemma:remv0} Let $\u$ be a nonnegative solution of \eqref{eq:main} in $B_1\setminus\{0\}$ with $\alpha = \frac{n}{n-2}$. Suppose further that $\u$ satisfies 
\begin{equation}\label{eq:remv0}
|\u(x)| = o(|x|^{n-2}(-\log |x|)^{\frac{n-2}{2}} )\quad\text{as }x\ra 0.
\end{equation}
Then the origin is a removable singularity. 
\end{lemma}

\begin{proof} Under the assumption \eqref{eq:remv0}, we claim that 
\begin{equation}\label{eq:remv-re0}
|\u(x)| \leq c|x|^{2-n + \delta}\quad\text{in }B_{r_0}\setminus\{0\},
\end{equation}
for some small $\delta>0$,  where $c>1$ and $r_0>0$ may depend on $\u$. 

Consider the auxiliary function
\begin{equation*}
\vp_\e (x) = \left(C r_0^{-\delta} |x|^\delta + \e (-\log |x|)^{\frac{2-n}{2}} \right)|x|^{2-n}\quad\text{in }B_{r_0}\setminus\{0\},
\end{equation*}
where $C_0>0$ is the (universal) constant chosen from \eqref{eq:u-sup-0},  $r_0>0$ is a small radius to be determined later and $\e>0$ is an arbitrary small number. After some computations, one may verify that 
\begin{equation*}
\Delta \vp_\e \leq \frac{C_1}{|x|^2\log |x|} \vp_\e\quad\text{in }B_{r_0}\setminus\{0\},
\end{equation*}
by choosing $\delta,r_0>0$ small, $C_1>0$ large. Here one may choose $\delta$ and $C_1$ to depend only on $n$. 

Due to the assumption \eqref{eq:remv0}, we have $a(x) = |\u|^{\frac{2}{n-2}} = o (-|x|^2\log |x|)$, whence $\vp_\e$ becomes a supersolution of $\Delta u_i = -a(x) u_i$ in $B_{r_0}\setminus\{0\}$, by choosing $r_0>0$ sufficiently small, where $u_i$ is the $i$-th component of $\u$. The rest of the proof follows the same argument shown in the proof of Lemma \ref{lemma:remv+sub}, which  eventually leads us to 
$u_i \leq \vp_\e$ in $B_{r_0}\setminus\{0\}$.
Passing to the limit with $\e\ra 0$, we get 
\begin{equation*}
u_i(x) \leq C_0 r_0^{-\delta} |x|^{2-n+\delta}\quad\text{in }B_{r_0}\setminus\{0\}.
\end{equation*}
Now that this inequality holds for any $1\leq i\leq m$, we arrive at \eqref{eq:remv-re0} with $c = C_0r_0^{-\delta}\sqrt{m}$. 

Thus, it follows from \eqref{eq:remv-re} that $a(x) = |\u|^{\frac{2}{n-2}}\in L^{\frac{n}{2-\eta}}(B_1)$, for some $\eta>0$. We know from Lemma \ref{lemma:dist} that $u_i$ is a distribution solution of $-\Delta u_i = a(x) u_i$ in $B_1$ for each $1\leq i \leq m$. Hence, the classical result \cite[Theorem 1]{S} by Serrin implies that $u_i$ has a removable singularity at the origin, and the lemma is proved. 
\end{proof}

Theorem \ref{theorem:main} (iv) is now merely a combination of Lemma \ref{lemma:class0} and Lemma \ref{lemma:remv0}.

\begin{proof}[Proof of Theorem \ref{theorem:main} (iv)]
If $\u$ has a non-removable singularity at the origin, then according to Lemma \ref{lemma:remv0}, $\u$ does not satisfy \eqref{eq:remv0}. By Lemma \ref{lemma:class0}, we have \eqref{eq:asym0}, proving the theorem.  
\end{proof}



\noindent {\bf Acknowledgement.} S. Kim was supported by National Research Foundation of Korea (NRF) grant funded by the Korean government (NRF-2014-Fostering Core Leaders of the Future Basic Science Program). H. Shahgholian was supported in part by Swedish Research Council.

This work was partly conducted during the first and second author's visit to Royal Institute of Technology (KTH) in Stockholm. They wish to thank Henrik Shahgholian for the kind invitation and KTH for the hospitality.

The authors would like to thank the anonymous referees for their valuable comments. Especially, we are grateful for one of the referees who pointed out a precise characterization of the new Pohozaev invariant as well as the explicit solution featuring the nontrivial invariant in the two-particle system.


\end{document}